\pdfoutput=1
\documentclass[10pt]{article}
\usepackage{geometry}                
\usepackage[parfill]{parskip}    
\usepackage{graphicx}
\usepackage{amsmath}
\usepackage{amssymb}
\usepackage{epstopdf}
\usepackage{verbatim}

\newtheorem{theorem}{Theorem}
\newtheorem{lemma}[theorem]{Lemma}
\newtheorem{proposition}[theorem]{Proposition}
\newtheorem{corollary}[theorem]{Corollary}

\newenvironment{proof}[1][Proof]{\begin{trivlist}
\item[\hskip \labelsep {\bfseries #1}]}{\end{trivlist}}

\newcommand{\qed}{\hfill\ensuremath{\square}}
\newcommand*\conj[1]{\overline{#1}}

\title{Modified Energy Functionals and the NLS Approximation}
\author{Patrick Cummings, C. Eugene Wayne}
\date{}                                           

\begin{document}
\maketitle

\begin{abstract}
We consider a model equation from \cite{CEW} that captures important properties of the water wave equation. We give a new proof of the fact that wave packet solutions of this equation are approximated by the nonlinear Schr\"{o}dinger equation. This proof both simplifies and strengthens the results of \cite{CEW} so that the approximation holds for the full interval of existence of the approximate NLS solution rather than just a subinterval. Furthermore, the proof avoids the problems associated with inverting the normal form transform in \cite{CEW} by working with a modified energy functional motivated by \cite{Craig} and \cite{Hunter}.
\end{abstract}

\section{Introduction}

The 2D water wave problem in the case of finite depth and no surface tension studies the irrotational flow of a homogeneous, inviscid, irrotational fluid in a canal of finite depth and infinite length under the effects of gravity. It has been shown that under these assumptions the evolution of the system is determined by the evolution of the position of the surface and the horizontal velocity at the surface.

It was shown in \cite{Dull} that there exist small solutions that can be approximated by the ansatz 
\begin{equation}
\epsilon\Psi_{NLS} = \epsilon A(\epsilon (x+c_g t),\epsilon^2 t)e^{i(k_0x+\omega_0 t)} \phi(k_0) \label{psidef}\end{equation} 
 where $A$ is a solution of the NLS equation 
\begin{equation} \partial_T A = i \nu_1 \partial_X^2 A + i \nu_2 A |A|^2.  \label{NLS}\end{equation}
The fact that the NLS equation should approximate the evolution of wave packets on a fluid surface was first predicted by V.E. Zakharov in \cite{zakharov}. This equation describes slow modulations in time and space of the temporally and spatially oscillating wave train $e^{i(k_0x+\omega_0t)}.$  We have the coefficients $\nu_j = \nu_j(k_0) \in \mathbb{R},$ $A(X,T) \in \mathbb{C}$ is the complex-value amplitude,  $0 < \epsilon \ll 1$ is a small perturbation parameter, and $\phi(k_0) \in \mathbb{C}^2$ an eigenvector for the linearized water wave equation. The slow spatial scale is $X=\epsilon (x+c_gt) \in \mathbb{R}$ and the slow time scale is $T=\epsilon^2 t.$ The basic spatial wave number $k_0$ and the basic temporal wave number $\omega_0$ are related via the linear dispersion relation of the water wave problem \[ \omega(k)^2 = k\tanh(k) .  \] 
Finally, the group velocity $c_g$ is given by $\partial_k\omega |_{k=k_0,\omega=\omega_0}.$

In this paper we consider the second-order equation
\begin{equation}\label{PDE} \partial_t^2 u = - \omega^2 u  - \omega^2 (u^2)\end{equation}
 with \(\widehat{\omega u} = \text{sgn} (k)\sqrt{k \tanh(k)}\cdot\widehat{u}(k).\) This equation was studied in \cite{CEW} as a model problem of the 2D water wave problem. By choosing $\omega$ this way, the model problem and the 2D water wave problem described above have the same linear dispersion relation. They also have many of the same difficulties that arise when proving the validity of the NLS approximation. Those difficulties are: a quasilinear equation, a quadratic nonlinearity, the trivial resonance at the wave number $k=0,$ and the nontrivial resonance at $k=k_0.$

The essential difference and advantage in this model problem is the relative simplicity of the linear and nonlinear term. In the case of the water wave problem, both terms are much more involved and include the Dirichlet-Neumann operator. As in Schneider and Wayne \cite{CEW} and D\"{u}ll et al \cite{Dull}, the goal of the model problem is to present a method which can then be used again on the 2D water wave problem. In that light, we wish to show that the model equation (\ref{PDE}) can be approximated by the ansatz (\ref{psidef}). 

On the basis of the form of the ansatz one expects an approximation result to hold for times $\mathcal{O}(\epsilon^{-2}).$ While \cite{CEW} and subsequently \cite{Dull} gave a result that held on the correct qualitative time interval, the result was linked to a specific time $T_1/\epsilon^2.$ One would like an arbitrary time $T_0/\epsilon^2$ with $T_0$ coming from the approximation $\Psi_{NLS}.$ Unlike the qualitative interval, if an arbitrary time can be proven, then we will be able to increase the interval for the approximation by increasing the interval of existence for $\Psi_{NLS}.$ The new result holds in this kind of time interval. We further show that there is an open ball of solutions of (\ref{PDE}) that start close to the ansatz (\ref{psidef}), remain close to it for the relevant time interval, and are unique. Moreover, this result is done in $H^2$ as opposed to the analytic norms used in \cite{CEW}. 

The improvements to this approximation result are the result of incorporating new ideas recently put forth in the literature. In particular, we use the modified energy method from Hunter et al in \cite{Hunter} and another modified energy result from Craig in \cite{Craig}. We also use some of the space-time resonance method of Germain-Masmoudi-Shatah introduced in \cite{Germain}. The first use of Hunter's modified energy method to prove an approximation result of the type discussed here is due to D\"{u}ll \cite{Dull16}, who studied a quasilinear wave equation in a case with no resonances. Very recently, D\"{u}ll and He{\ss} \cite{DullHeb} used the same method to study a different quasilinear dispersive equation with non-trivial resonances similar to those encountered in our equation. 

Before stating our result, we define the operator $\Omega$ by $\widehat{\Omega u} (k) = i \widehat{\omega u} (k).$ This operator $\Omega$ has two properties we wish to take advantage of. The first is that $\Omega$ is anti-symmetric, second, is that if $u$ is real, then $\Omega u$ is real as well. The fact that $\Omega$ preserves the real-valuedness of functions will be pivotal to many of the cancellations used throughout this paper.

With this new operator we can rewrite (\ref{PDE}) as the first order system

\begin{equation}\label{system} \partial_t \begin{pmatrix} u \\ v \end{pmatrix} =  \begin{pmatrix} 0 & \Omega \\ \Omega & 0 \end{pmatrix}  \begin{pmatrix} u \\ v \end{pmatrix} + \begin{pmatrix} 0  \\  \Omega(u^2) \end{pmatrix} .   \end{equation}

Using this notation we have the following theorem.

\begin{theorem} \label{eq:main}
For all $k_0 > 0$ and for all $C_1,T_0 > 0$ there exist $C_2 > 0, \epsilon_0> 0 $ such that for all solutions $A \in C([0,T_0],H^6(\mathbb{R},\mathbb{C}))$ of the NLS equation (\ref{NLS}) with 
\[   \sup_{T \in [0,T_0]}   \|A(\cdot, T)\|_{H^6}  \leq C_1 \] the following holds. For all $\epsilon \in (0,\epsilon_0)$ there exists an open set of initial conditions $f,g \in H^2$ such that
there exists a unique solution of (\ref{system}) which satisfies 
\[\begin{aligned}
& \sup_{t \in [0,T_0/\epsilon^2]}\left\| \begin{pmatrix}u\\v \end{pmatrix}(\cdot, t) - \epsilon \Psi_{NLS}(\cdot,t) \right\|_{(C_b^0(\mathbb{R},\mathbb{R}))^2} \leq C_2 \epsilon^{3/2} , \\
&\hspace{.2 in} u(x,0) =\epsilon \Psi_{NLS,1}(x,0)+ f(x), \\
&\hspace{.2 in}  v(x,0)=\epsilon \Psi_{NLS,2}(x,0) +g(x), 
\end{aligned}\] 
where $\phi(k_0)$ in the definition of $\epsilon \Psi_{NLS}$ in (\ref{psidef}) can be chosen either as $\begin{pmatrix} 1 \\ 1 \end{pmatrix}$ or $\begin{pmatrix} 1 \\ -1 \end{pmatrix}$
\end{theorem}

As mentioned above, the first (non-rigorous) derivation of the NLS equation as an approximate equation for the evolution of wave-packets on fluids was due to Zakharov in \cite{zakharov}.  The first rigorous investigation of the NLS equation in the context of water waves was due to Craig, Sulem, and Sulem in \cite{CSS} which established the NLS equation as an approximate equation for water waves, but only over a very short time interval - too short for the characteristic phenomena (e.g. solitons) of NLS to manifest themselves.  A general approach to justifying modulation equations like NLS was developed by Kalyakin in \cite{Kalyakin}, who used normal-forms as well as averaging methods.  However, his results did not extend to the sort of quasi-linear PDE’s considered here.  Much closer in spirit to the present work is the paper of Kirrman, Schneider, and Mielke in \cite{Kirrman} who gave a general approach to justifying NLS approximations and applied it to nonlinear PDEs with cubic nonlinear terms.  This was then extended by Schneider in \cite{Schneider} in the case of quadratic nonlinearities via a normal form method under a non-resonance condition for the nonlinearity.  Further developments lead to the paper \cite{CEW}, which was the motivation for the present work, and its extension by D\"{u}ll, Schneider and Wayne, \cite{Dull},  who proved that the NLS equation could be used to approximate wave packets on the surface of an inviscid, irrotational fluid in a channel of finite depth.  The NLS approximation for water waves on a fluid of infinite depth was treated in the two-dimensional case by Totz and Wu in \cite{Totz} and, more recently by Totz in \cite{Totz3D} for the three-dimensional problem.  Interestingly, the methods needed to treat the cases of finite and infinite depth seem to be quite different.  We hope that the methods of the present paper will extend to the water wave problem as well.

The structure of this paper is as follows. In Section 2 we set up the equations as well as give estimates on the residual. In Section 3 we define the energy and describe the underlying ideas of the paper. In Section 4 we look at the evolution of the energy. In particular, we separate the energy into three pieces, each of which will be dealt with differently. In Section 5, we work with the three pieces and use some of the space-time resonance methods developed by Germain-Shatah-Masmoudi. Finally, in Section 6 we conclude the energy estimates via a Gronwall-type argument. 

\textbf{Notation.} We denote the Fourier transform by $\widehat{u}(k) = \frac{1}{2\pi} \int u(x)e^{-ikx}dx.$ The Sobolev space $H^s$ is equipped with the norm $\|u\|^2_{H^s} = \int (1+|k|^2)^s |\widehat{u}(k)|^2 dk.$ Let $\|u\|_{C^b_0} = \sup_{x \in \mathbb{R}} |u(x)|.$ We also write $A \lesssim B$ if there exists a constant $C,$ independent of $\epsilon,$ such that $A \leq CB.$


\section{Setup equations and Residual Estimates}

At first glance, we should define the error $R$ using\[  \begin{pmatrix} u \\ v \end{pmatrix} = \epsilon \begin{pmatrix} \tilde{\Psi}_1 \\ \tilde{\Psi}_2 \end{pmatrix} +  \epsilon^\beta \begin{pmatrix} R_1 \\ R_2 \end{pmatrix}  \]
where we have written $\tilde{\Psi}$ in place of $\Psi_{NLS}$ for convenience. 
Inserting this approximation into (\ref{system}) we obtain the following equation for $R$
\begin{equation} \partial_t \begin{pmatrix} R_1 \\ R_2 \end{pmatrix}  = \begin{pmatrix} 0 & \Omega \\ \Omega & 0 \end{pmatrix}  \begin{pmatrix} R_1 \\ R_2 \end{pmatrix}  + 2 \epsilon \begin{pmatrix} 0 \\ \Omega(\tilde{\Psi}_1 R_1 ) \end{pmatrix} + \epsilon^\beta \begin{pmatrix} 0 \\ \Omega(R_1^2 ) \end{pmatrix} + \epsilon^{-\beta} \text{ Res}(\epsilon\tilde{\Psi}), \label{approxnotheta} \end{equation}
where we have defined 
\[ \text{Res}(\epsilon \tilde{\Psi}) = -\epsilon\partial_t \begin{pmatrix}\tilde{\Psi}_1 \\ \tilde{\Psi}_2 \end{pmatrix} + \epsilon\begin{pmatrix} 0 & \Omega \\ \Omega & 0 \end{pmatrix} \begin{pmatrix}\tilde{\Psi}_1 \\ \tilde{\Psi}_2 \end{pmatrix} +\epsilon^2\begin{pmatrix} 0 \\ \Omega\left( \tilde{\Psi}_1^2 \right) \end{pmatrix}.  \]

To show that $R$ is small for times of $\mathcal{O}(\epsilon^{-2}),$ we must tackle multiple issues. The first is the quasilinearity. If we look at the evolution of the $H^2$-norm of $R,$ then we see that the quasilinearity of the problem makes it so that the best bound we can get is a multiple of $\|R\|_{H^{5/2}}$ which prevents the estimates from closing. The next issue is the order of $\epsilon.$ Again if we look at the evolution of $\|R\|_{H^2},$ we can modify our ansatz $\Psi$ in such a way that the residual is $\mathcal{O}(\epsilon^{11/2}).$ Then choosing $2 \leq \beta \leq 7/2,$ the third and fourth terms are of high enough order in terms of $\epsilon.$ However, the second term poses the biggest problem. It is only $\mathcal{O}(\epsilon)$ and a direct application of Gronwall's inequality will only control growth for times of $\mathcal{O}(\epsilon^{-1}).$ We will use the method of space-time resonances as well as a modified energy method to overcome this. Both methods will require us to avoid two resonances, one at $k=0$ and one at $k=\pm k_0.$ The former will be bounded using the form of the nonlinearity with a so-called transparency condition and modifications to the energy. For the latter we will use a weight function $\vartheta,$ first introduced in \cite{CEW}, that also takes advantage of the fact that the nonlinearity vanishes near $k=0.$ 

As mentioned above, to bound the residual term in terms of epsilon, we will need to modify the ansatz. In \cite{CEW} it was shown this could be done by adjusting the approximation $\Psi_{NLS}$ with higher order terms. This gives us a new approximation $\Psi = \Psi_{NLS} + \mathcal{O}(\epsilon)$ with many critical properties for the following estimates. We specify in more detail the advantages of $\Psi$ below but refer the reader to \cite{CEW} for the derivation of the higher order terms.

Now we define the weight function $\vartheta$ in Fourier space as 
\begin{equation}
\widehat{\vartheta}(k) = \begin{cases} 
1, & \text{ if } |k| > \delta \\ 
\epsilon + (1-\epsilon) |k|/\delta, & \text{ if } |k| \leq \delta.   \end{cases}
\label{vartheta} \end{equation}

We can then define $R$ using our new ansatz $\Psi$ and $\vartheta$ as  \[  \begin{pmatrix} u \\ v \end{pmatrix} = \epsilon \begin{pmatrix} \Psi_1 \\ \Psi_2 \end{pmatrix} +  \epsilon^\beta \vartheta \begin{pmatrix} R_1 \\ R_2 \end{pmatrix}  \] 
Here we have abused notation slightly by defining $\widehat{\vartheta R } = \widehat{\vartheta}\widehat{R}$ as opposed to writing $\vartheta * R.$ Similar to $\Omega,$ if $R$ is real, then $\vartheta R $ is real as well. Inserting this approximation into (\ref{system}) we obtain the following equation for $R$
\begin{equation} \partial_t \begin{pmatrix} R_1 \\ R_2 \end{pmatrix}  = \begin{pmatrix} 0 & \Omega \\ \Omega & 0 \end{pmatrix}  \begin{pmatrix} R_1 \\ R_2 \end{pmatrix}  + 2 \epsilon \vartheta^{-1}\begin{pmatrix} 0 \\ \Omega(\Psi_1 \vartheta R_1 ) \end{pmatrix} + \epsilon^\beta \vartheta^{-1}\begin{pmatrix} 0 \\ \Omega\left((\vartheta R_1)^2 \right) \end{pmatrix} + \epsilon^{-\beta} \vartheta^{-1}\text{ Res}(\epsilon\Psi). \label{approx} \end{equation}

We diagonalize the system using 
\[ \begin{pmatrix} F_1 \\ F_2 \end{pmatrix} = S \begin{pmatrix} R_1 \\ R_2 \end{pmatrix}, \begin{pmatrix} G_1 \\ G_2 \end{pmatrix} = S \begin{pmatrix} \Psi_1 \\ \Psi_2 \end{pmatrix}, \text{with } S = S^{-1} =\frac{1}{\sqrt{2}}\begin{pmatrix} 1 & 1 \\ 1 & -1\end{pmatrix} . \]
This gives us
\[ \partial_t F = \Lambda F + 2 \epsilon \vartheta^{-1} N(G,\vartheta F) + \epsilon^\beta \vartheta^{-1} N(\vartheta F,\vartheta F) + \epsilon^{-\beta}   \vartheta^{-1} \text{Res}(\epsilon G)  \]
with
 \begin{align}& \Lambda = \begin{pmatrix} \Omega & 0 \\ 0 & -\Omega \end{pmatrix}= \begin{pmatrix} \Omega_1 & 0 \\ 0 & \Omega_2 \end{pmatrix} = \begin{pmatrix} i\omega_1 & 0 \\ 0 & i\omega_2 \end{pmatrix}, \nonumber\\
&\widehat{N}_j(\widehat{U},\widehat{V}) (k)= \sum_{m,n=1}^2 \int \frac{i\omega_j(k)}{\sqrt{2}} \widehat{U_m}(k-\ell)\widehat{V_n}(\ell) d \ell, \label{nonlinear} \end{align}
and \[\hspace{.3 in}\text{Res}(\epsilon G) = S\cdot \text{Res}(\epsilon \Psi) =  -\epsilon\partial_t G + \epsilon\Lambda G + \epsilon^2 N(G,G),\] which will satisfy the same estimates as $\text{Res}(\epsilon\Psi).$

In order to apply the space-time methods, it will sometimes be easier to move into a rotating coordinate frame. We define 
\begin{equation}  \begin{pmatrix} f_1 \\ f_2 \end{pmatrix} = e^{-\Lambda t} \begin{pmatrix} F_1 \\ F_2\end{pmatrix}\text{ and }   \begin{pmatrix} g_1 \\ g_2 \end{pmatrix} = e^{-\Lambda t} \begin{pmatrix} G_1 \\ G_2 \end{pmatrix}.   \label{rotatingeqn}
\end{equation} 

For these variables we have 
\begin{equation}
 \partial_t f =  2 \epsilon\,  e^{-\Lambda t} \vartheta^{-1} N(e^{\Lambda t}  g,e^{\Lambda t} \vartheta f) + \epsilon^\beta\,  e^{-\Lambda t}\vartheta^{-1} N(e^{\Lambda t} \vartheta f,e^{\Lambda t}\vartheta f) + \epsilon^{-\beta}  \, e^{-\Lambda t} \vartheta^{-1} \text{Res}(\epsilon G).
\end{equation}

Note that the $\mathcal{O}(1)$ term no longer appears in this coordinate frame.

We now specify in more detail the properties of the higher order terms in $\Psi.$ In \cite{CEW} and \cite{SchneiderJAF}, it was shown that making two modifications to $\Psi_{NLS}$ did not cause any significant changes to the approximation result. The first modification was the addition of terms of the form
\[ \tilde{\psi}^{j_1}_{j_2,j} = \epsilon^{j_1} A^{j_1}_{j_2,j}(\epsilon(x+c_gt), \epsilon^2t) e^{ij_2(k_0x+\omega_0t)}  \]
where $j_1 \geq 1$ and $j_2 =-4,-3,...,4.$ These terms are of higher order in $\epsilon$ than $\Psi_{NLS}$ and in Fourier space are concentrated near $j_2k_0.$ In fact, these terms, as well as the original $\Psi_{NLS},$ can be cut-off in Fourier space via the new definition
\[  \widehat{\psi}^{j_1}_{j_2,j}(k) = \begin{cases} \widehat{\tilde{\psi}}^{j_1}_{j_2,j}(k) &\text{if } |k-j_2k0| < \delta,\\ 0 & \text{otherwise} \end{cases}   \]
where $\delta>0$ is small but independent of $\epsilon$. Making both of these changes will only change the approximation up to order $\epsilon$ as shown in the following lemma.


\begin{lemma}
(Lemma 5 in \cite{CEW}) Let $A \in C([0,T_0],H^6(\mathbb{R},\mathbb{C}))$ be a solution of the NLS equation (\ref{NLS}) with 
\[   \sup_{T \in [0,T_0]}   \|A(\cdot, T)\|_{H^6}  \leq C_1. \]  Then for all $\sigma,r > 0$ an approximation $\Psi$ exists for all $T \in [0,T_0]$ such that 
\[ \begin{aligned}
\sup_{T \in [0,T_0]} \|\Psi(T)- \Psi_{NLS}(T)\|_{C_b^0} &\leq C\epsilon\\
\sup_{T \in [0,T_0]} \|\mbox{Res}(\epsilon \Psi(T))\|_{Y^2_{\sigma, r}} &\leq C \epsilon^{11/2}.
  \end{aligned} \]
\end{lemma}
The norm above $\| \cdot \|_{Y^p_{\sigma,r}}$ is given by 
\[  \| \widehat{u} \|_{Y^p_{\sigma,r}} = \|e^{\sigma |k|} (1+k^2)^{r/2}\widehat{u}  \|_{L^p}.\] 
In particular, we have the following corollary.

\begin{corollary}
For any $r \geq 0,$ there exists an approximation $\Psi(X,T),$ to lowest order given by (\ref{NLS}) and a constant $ C_r > 0$ such that 
\[ \begin{aligned}
\sup_{T \in [0,T_0]} \|\Psi(T)- \Psi_{NLS}(T)\|_{C_b^0} &\leq C_r\epsilon\\
\sup_{T \in [0,T_0]} \|\text{Res}(\epsilon \Psi(T))\|_{H^r} &\leq C_r \epsilon^{11/2}\\
\sup_{T \in [0,T_0]} \|\vartheta^{-1}\text{Res}(\epsilon \Psi(T))\|_{H^r} &\leq C_r \epsilon^{5}.
  \end{aligned} \]
\end{corollary}

This last estimate comes from Corollary 19 in \cite{CEW} and uses the fact that the residual is actually $\mathcal{O}(\epsilon^6)$ in $L^{\infty}.$ We recall that although this requires that we add higher order terms to our approximation, the lowest order terms remain those from $\Psi_{NLS}$ cutoff away from $\pm k_0.$ In particular, the lowest order terms are concentrated around $\pm k_0$ such that we can write
\[\begin{aligned}
 \widehat{\Psi}(k) &= \widehat{\Psi}^c(k) + \epsilon \widehat{\Psi}^s(k) \\
 \widehat{G}(k) &= \widehat{G}^c(k) + \epsilon \widehat{G}^s(k) 
\end{aligned} \]
with \[ \text{supp}(\widehat{\Psi}^c) =\text{supp}(\widehat{G}^c) \subset \left\{ k \;\big|\; |k \pm k_0| < \delta \right\} \] due to the fact that we cut the approximation off in Fourier space. Moreover, the bounds for both parts $\widehat{G}^c$ and $\widehat{G}^s$ do not grow with time.


\section{The Energy}

As mentioned earlier, the two main issues at play are the resonances and the quasilinearity. Since we will deal with both in different ways, it will be easier to analyze the energy if we split it into pieces and work with each separately. In that light, we define the energy as 
\[E = E_0 + E_1 +E_2\]
with 
\[\begin{aligned}
&E_0 = \int  R_1^2 + R_2^2 + (\partial_x^2 R_1)^2  + (\partial_x^2 R_2)^2  \; dx\\
 &E_1 = \int 2\epsilon \Psi_1 ( \partial^2_x R_1 )^2 + 2\epsilon^\beta \vartheta R_1  (\partial^2_x  R_1)^2 \;dx \\  
&E_2 = \sum_{j=1}^2 \int\displaylimits_{|k| < \delta}   \epsilon \conj{\widehat{F}_j} B_j(\widehat{G}^c, \widehat{F})+ \epsilon \widehat{F}_j\conj{ B_j(\widehat{G}^c, \widehat{F})} + 2\epsilon^2\conj{B_j(\widehat{G}^c, \widehat{F})}B_j(\widehat{G}^c, \widehat{F})\; dk.
\end{aligned}\]
We have defined 
\[B_j(\widehat{G}^c, \widehat{F}) = \frac{1}{\sqrt{2}}\sum_{m,n=1}^2 \int \frac{\omega_j(k)}{\phi^{j}_{mn}(k,\ell)}\frac{\widehat{\vartheta}(\ell)}{\widehat{\vartheta}(k)} \widehat{G}_m^c(k-\ell)\widehat{F}_n(\ell) d\ell   \]
and 
\[ \phi^j_{mn}(k,\ell) = -\omega_j(k) + \omega_m(k-\ell) + \omega_n(\ell).\]

The terms in $E_0$ are equivalent to the $H^2$ norm. The terms in $E_1$ are chosen to counteract the effects of the quasilinearity. The evolution of these terms will cancel with the terms with the most derivatives from the $H^2$ norm. This idea comes from W. Craig in \cite{Craig}. The terms in $E_2$ are chosen via the normal-form method. This idea is similar to that used in \cite{Hunter} by Hunter et al. In our case, the bilinear transformation $B$ comes from the transformation that would remove the quadratic terms in (\ref{approx}), first derived in \cite{CEW}. However, unlike the usual idea of the normal-form method, we will only use the transformation $B$ to cancel quadratic terms in a very small region. In particular, we use it to cancel those terms with the least derivatives and only in Fourier space near $k=0.$ This will allow us to avoid complicated problems that come with inverting the full transformation $B.$ For the other terms in that region, as well as those outside we will use the method of space-time resonances.

We note that 
\[\begin{aligned}
 E_1 &= \int 2\epsilon \Psi_1 ( \partial^2_x R_1 )^2 + 2\epsilon^\beta  \vartheta R_1  (\partial^2_x  R_1)^2 \;dx \\ 
& \lesssim \epsilon \|\Psi\|_{L^{\infty}} \|R\|_{H^2}^2 + \epsilon^{\beta}   \|R\|_{H^2}^3.
\end{aligned}\]
Furthermore, by Proposition \ref{prop:modenergy} in the Appendix we know for any $p>2,$ there exists a $\gamma > 0$ such that
\[\begin{aligned}
E_2 &=\sum_{j=1}^2 \int\displaylimits_{|k| < \delta}   \epsilon \conj{\widehat{F}_j} B_j(\widehat{G}^c, \widehat{F})+ \epsilon \widehat{F}_j\conj{ B_j(\widehat{G}^c, \widehat{F})} + 2\epsilon^2\conj{B_j(\widehat{G}^c, \widehat{F})}B_j(\widehat{G}^c, \widehat{F})\; dk \\
& \lesssim \delta^{\gamma} \left(\|\widehat{A}\|_{L^p} + \|\widehat{A}\|_{L^p}^2 \right) \|F\|_{L^2}^2. \end{aligned}\]
Since $\|\widehat{A}\|_{L^p} < C$ for $t \in [0,T_0],$ for $\epsilon, \delta>0$ small enough, we have 
\[ (1-C_1 \epsilon - C_2\delta^{\gamma}) E \leq \|R \|_{H^2}^2 \leq (1+C_1 \epsilon + C_2\delta^{\gamma})E   \]
and so the energy above is equivalent to the $H^2$ norm of $R.$

Since we hope to show that $\|R\|_{H^2} \leq C$ independent of $\epsilon$ for the time interval $ \left[0,\frac{T_0}{\epsilon^2}\right],$ the goal of the rest of the paper will be to bound the evolution of $E$ sufficiently small. We look to show that \[ \partial_t E \lesssim \epsilon^2(1+E) + \epsilon^3E^2.  \] Then an application of Gronwall's inequality will be sufficient to conclude that 
\[ \sup_{t \in [0,T_0/\epsilon^2]} E(t) \leq C.\]
 In that light, we will refer to terms throughout as ``bounded" if they can be bounded by a constant multiple of $\epsilon^2(1+E) + \epsilon^3E^2.$


\section{Evolution of $E$}

We now look at the evolution of our energy. We calculate
\[\begin{aligned} 
\frac{1}{2} \partial_t E_0 &= \int   R_1 \left( \Omega R_2 + \epsilon^{-\beta} \vartheta^{-1}\text{Res}_1 (\epsilon \Psi)\right)  \\
+& R_2 \left( \Omega R_1 + 2\epsilon \vartheta^{-1} \Omega(\Psi_1 \vartheta R_1 ) + \epsilon^{\beta} \vartheta^{-1} \Omega\left((\vartheta R_1)^2\right) + \epsilon^{-\beta}\vartheta^{-1} \text{Res}_2(\epsilon \Psi) \right) \\
+&\partial_x^2 R_1 \left( \Omega \partial_x^2 R_2 + \epsilon^{-\beta} \partial_x^2 \vartheta^{-1} \text{Res}_1 (\epsilon \Psi)\right)  \\
+& \partial_x^2 R_2 \left( \Omega \partial_x^2 R_1 + 2\epsilon \vartheta^{-1} \Omega \partial_x^2(\Psi_1 \vartheta R_1 ) + \epsilon^{\beta}  \vartheta^{-1}\Omega \partial_x^2\left((\vartheta R_1)^2\right) + \epsilon^{-\beta} \vartheta^{-1}\partial_x^2 \text{Res}_2(\epsilon \Psi) \right) \;dx\\
\frac{1}{2} \partial_t E_1 & = \int \epsilon \left( \Omega \Psi_2 - \epsilon^{-1}\text{Res}_1(\epsilon \Psi) \right) (\partial_x^2R_1)^2  \\
+& 2\epsilon \Psi_1 \cdot\partial_x^2 R_1 \left( \Omega\partial_x^2 R_2 + \epsilon^{-\beta} \partial_x^2\vartheta^{-1} \text{Res}_1 (\epsilon \Psi)\right) \\
+& \epsilon^\beta \left( \Omega \vartheta R_2 + \epsilon^{-\beta}  \text{Res}_1 (\epsilon \Psi)\right)  \left(\partial_x^2  R_1\right)^2 \\
+ & 2\epsilon^\beta \vartheta R_1\cdot \partial_x^2 R_1 \left( \Omega \partial_x^2  R_2 + \epsilon^{-\beta} \partial_x^2 \vartheta^{-1} \text{Res}_1 (\epsilon \Psi)\right)\;dx \\
\frac{1}{2} \partial_t E_2 & =\frac{1}{2}\sum_{j=1}^2 \int\displaylimits_{|k| < \delta}    \epsilon \conj{\big(  \Omega_j\widehat{F}_j + 2 \epsilon \widehat{\vartheta}^{-1} \widehat{N}_j(\widehat{G}, \widehat{\vartheta} \widehat{F}) + \epsilon^\beta \widehat{\vartheta}^{-1} \widehat{N}_j (\widehat{\vartheta} \widehat{F}, \widehat{\vartheta} \widehat{F})} \\ 
+&  \conj{ \epsilon^{-\beta} \widehat{\vartheta}^{-1} \widehat{\text{Res} (\epsilon G)}    \big) } B_j(\widehat{G}^c, \widehat{F}) \\
+&   \epsilon \conj{\widehat{F}_j} B_j(\widehat{G}^c, \Lambda \widehat{F} + 2 \epsilon \widehat{\vartheta}^{-1} \widehat{N}(\widehat{G}, \widehat{\vartheta} \widehat{F}) + \epsilon^\beta \widehat{\vartheta}^{-1} \widehat{N} (\widehat{\vartheta} \widehat{F}, \widehat{\vartheta} \widehat{F}) + \epsilon^{-\beta}  \widehat{\vartheta}^{-1}\text{Res} (\epsilon G) ) \\
 + &\epsilon \conj{\widehat{F}_j} B_j(\partial_t \widehat{G}^c, \widehat{F}) \\
+& 2\epsilon^2\conj{B_j(\partial_t \widehat{G}^c, \widehat{F})}B_j(\widehat{G}^c, \widehat{F}) \\
+& 2\epsilon^2\conj{B_j\left( \widehat{G}^c, \Lambda \widehat{F} + 2 \epsilon \widehat{\vartheta}^{-1} \widehat{N}(\widehat{G}, \widehat{\vartheta} \widehat{F}) + \epsilon^\beta \widehat{\vartheta}^{-1} \widehat{N} (\widehat{\vartheta} \widehat{F}, \widehat{\vartheta} \widehat{F}) + \epsilon^{-\beta} \widehat{\vartheta}^{-1} \widehat{\text{Res} (\epsilon G)}  \right)}B_j(\widehat{G}^c, \widehat{F}) \\
+&  c.c. \; dk
\end{aligned}\]

where we write $c.c.$ for complex conjugate. This can be rewritten as

\[\begin{aligned} 
&\frac{1}{2} \partial_t E_0 = \int   R_1 \cdot \Omega R_2 + R_2 \cdot\Omega R_1 + \partial_x^2R_{1}\cdot \Omega \partial_x^2 R_2 + \partial_x^2 R_2\cdot \Omega \partial_x^2 R_1 \; dx\\
&\hspace{.5 in}+2\epsilon \int   R_2 \cdot \vartheta^{-1} \Omega (\Psi_1 \vartheta R_1) + \partial_x^2 R_2 \cdot \vartheta^{-1} \Omega\partial_x^2 (\Psi_1 \vartheta R_1)\;dx \\ 
&\hspace{.5 in} +\epsilon^{\beta}  \int R_2 \cdot \vartheta^{-1}\Omega\left( (\vartheta R_1)^2\right) + \partial_x^2R_2 \cdot \vartheta^{-1} \Omega\partial_x^2\left( ( \vartheta R_1)^2\right)  \;dx \\
\end{aligned}\]
\[\begin{aligned} 
&\frac{1}{2} \partial_t E_1 = 2\epsilon \int  \frac{1}{2} \Omega \Psi_2   (\partial_x^2R_1)^2 +\Psi_1 \cdot\partial_x^2 R_1 \cdot \Omega\partial_x^2 R_2  \;dx \\
&\hspace{.5 in} +\epsilon^{\beta}  \int\Omega \vartheta R_2  \left(\partial_x^2 R_1\right)^2 + 2 \vartheta R_1\cdot \partial_x^2  R_1 \cdot\Omega \partial_x^2  R_2 \;dx\\
&\frac{1}{2}\partial_t E_2 = \sum_{j=1}^2  \frac{1}{2}\epsilon \int\displaylimits_{|k| < \delta} 
\conj{\Omega_j \widehat{F}_j} \cdot B_j(\widehat{G}^c,\widehat{F})    +    \conj{\widehat{F}_j}\cdot B_j( \widehat{G}^c, \Lambda \widehat{F}) \\
&\hspace{1.2 in}+ \conj{\widehat{F}_j} \cdot B_j(\Lambda \widehat{G}^c, \widehat{F}) +  \conj{\widehat{F}_j}\cdot  B_j((\partial_t-\Lambda) \widehat{G}^c, \widehat{F}) \;dk  \\
&\hspace{.5 in}  +\epsilon^2  \int\displaylimits_{|k| < \delta}  \conj{ \widehat{\vartheta}^{-1} N(\widehat{G}, \vartheta \widehat{F})} \cdot B_j(\widehat{G}^c,\widehat{F})  + \conj{\widehat{F}_j} \cdot B_j(\widehat{G}^c,  \widehat{\vartheta}^{-1} N(\widehat{G}, \widehat{\vartheta} \widehat{F}) )\\
&\hspace{1.2 in} + \conj{B_j(\Lambda \widehat{G}^c, \widehat{F})} \cdot B_j(\widehat{G}^c, \widehat{F})+ \conj{B_j(( \partial_t- \Lambda) \widehat{G}^c, \widehat{F})} \cdot B_j(\widehat{G}^c, \widehat{F}) \\
& \hspace{1.2 in} + \conj{B_j( \widehat{G}^c, \Lambda\widehat{F})} \cdot B_j(\widehat{G}^c,  \widehat{F}) \; dk\\
&\hspace{.5 in}+  2\epsilon^3  \int\displaylimits_{|k|< \delta}\conj{B_j( \widehat{G}^c, \widehat{\vartheta}^{-1} N(\widehat{G}, \widehat{\vartheta} \widehat{F})   )}\cdot B_j(\widehat{G}^c, \widehat{F}) \;dk \\
&\hspace{.5 in} + c.c.
\end{aligned}\]

We have dropped terms from these expressions that involve $\text{Res}(\epsilon \Psi)$ and $\text{Res}(\epsilon G)$ as well as terms of order $\epsilon^{2+\beta}$ and higher in the expression for $\partial_t E_2$.  The residual terms can be bounded directly if we restrict $\beta \le 3$.  The higher order terms in $\epsilon$ can also be bounded by $\epsilon^2(1+E) + \epsilon^3E^2$  in a fashion similar to that used in Proposition 10.   At first sight it might appear that the terms of order $\epsilon^2$ and $\epsilon^3$ in the expression for $\partial_t E_2$ would obey similar bounds.  However, the factors of $\widehat{\vartheta}^{-1}$ in the integrand in these terms (either explicitly, or in the definition of $B_j$) are ${\cal O}(\epsilon^{-1})$ for $k \approx 0$, and this means that in fact, these terms are only of ${\cal O}(\epsilon)$ and must be treated more carefully.

For the rest of the terms, first notice that the $\mathcal{O}(1)$ terms cancel using the antisymmetry of $\Omega.$ Then the $\mathcal{O}(\epsilon)$ terms from $E_0$ and $E_1$ are of the form 

\begin{align}
2&\epsilon \int   R_2  \vartheta^{-1} \Omega (\Psi_1 \vartheta R_1) + \partial_x^2 R_2  \vartheta^{-1} \Omega\partial_x^2 (\Psi_1 \vartheta R_1)   +\frac{1}{2} \Omega \Psi_2   (\partial_x^2R_1)^2 +\Psi_1 \partial_x^2 R_1  \Omega\partial_x^2 R_2  \;dx \nonumber \\
=& 2\epsilon \int   R_2  \vartheta^{-1} \Omega (\Psi_1 \vartheta R_1)   +\frac{1}{2} \Omega \Psi_2   (\partial_x^2R_1)^2 + \partial_x^2 R_2  \vartheta^{-1} \Omega\partial_x^2 (\Psi_1 \vartheta R_1)  - \Omega (\Psi_1 \partial_x^2 R_1 ) \partial_x^2 R_2  dx \nonumber\\
 =& 2\epsilon \int   R_2  \vartheta^{-1} \Omega (\Psi_1 \vartheta R_1)   +\frac{1}{2} \Omega \Psi_2   (\partial_x^2R_1)^2 +  \partial_x^2 R_2  \vartheta^{-1} \Omega (\partial_x^2\Psi_1 \vartheta R_1)\nonumber\\
& \hspace{.1 in}+ 2\partial_x^2 R_2  \vartheta^{-1} \Omega (\partial_x\Psi_1 \vartheta \partial_x R_1)  + \partial_x^2 R_2  \vartheta^{-1} \Omega (\Psi_1 \vartheta \partial_x^2R_1) - \partial_x^2 R_2  \Omega (\Psi_1 \cdot\partial_x^2 R_1 )   \;dx. \label{epsterms}
\end{align}

We expand and rewrite the terms in this way to take advantage of a cancellation later. Note that $\partial_x^2 R_2 \cdot \vartheta^{-1} \Omega (\Psi_1 \vartheta \partial_x^2R_1)$ has two derivatives landing on each factor of $R$ as well as the presence of $\Omega$ on one of the factors of $R.$ Since we hope to bound things in terms of our energy, which is equivalent to the $H^2$ norm of $R,$ this term has too many derivatives and so poses the biggest threat. However, without the presence of $\vartheta,$ the last two terms above would cancel each other exactly. This ``cancellation'' is exactly the reason for the addition of $E_1$ to the energy. 

The $\mathcal{O}(\epsilon^\beta)$ terms have a similar expansion
\[\begin{aligned}
&\epsilon^{\beta}  \int R_2 \cdot \vartheta^{-1}\Omega\left( (\vartheta R_1)^2\right) + \partial_x^2R_2 \cdot \vartheta^{-1} \Omega\partial_x^2\left( ( \vartheta R_1)^2\right)+\Omega \vartheta R_2  \left(\partial_x^2 R_1\right)^2 \\
&\hspace{1 in}+ 2 \vartheta R_1\cdot \partial_x^2  R_1 \cdot\Omega \partial_x^2  R_2 \;dx\\
& = \epsilon^{\beta}  \int R_2 \cdot \vartheta^{-1}\Omega\left( (\vartheta R_1)^2\right) +\Omega \vartheta R_2  \left(\partial_x^2 R_1\right)^2\\
& \hspace{1 in}+ 2\partial_x^2R_2 \cdot \vartheta^{-1} \Omega\left(\vartheta R_1\cdot \vartheta\partial_x^2R_1 +  (\vartheta\partial_xR_1)^2\right)  - 2\Omega\left( \vartheta R_1\cdot \partial_x^2  R_1 \right) \partial_x^2  R_2 \;dx\\
& = \epsilon^{\beta}  \int R_2 \cdot \vartheta^{-1}\Omega\left( (\vartheta R_1)^2\right) +\Omega \vartheta R_2  \left(\partial_x^2 R_1\right)^2 + 2\partial_x^2R_2 \cdot \vartheta^{-1} \Omega\left( (\vartheta\partial_xR_1)^2\right)  \\
& \hspace{1 in}+ 2\partial_x^2R_2 \cdot \vartheta^{-1} \Omega\left(\vartheta R_1\cdot \vartheta\partial_x^2R_1\right) - 2\partial_x^2  R_2 \cdot \Omega\left(\vartheta R_1\cdot \partial_x^2 R_1 \right)  \;dx\\
&\lesssim  \epsilon^{\beta} E^{3/2}.
\end{aligned}\]

Above we have used Cauchy-Schwartz and the fact that $H^2$ is a Banach algebra. We also used the fact that the last two terms here exhibit the sort of cancellation alluded to for the $\mathcal{O}(\epsilon)$ terms. In particular, by Plancharel
\[\begin{aligned}
2&\epsilon^\beta \int \conj{\partial_x^2R_2} \cdot \vartheta^{-1} \Omega\left(\vartheta R_1\cdot \vartheta\partial_x^2R_1\right)- \conj{\partial_x^2  R_2}\Omega\left( \vartheta R_1\cdot \partial_x^2  R_1 \right) \;dx\\
& = \sqrt{2}\epsilon^\beta \int \left(\frac{\widehat{\vartheta}(k-\ell)\widehat{\vartheta}(\ell)}{\widehat{\vartheta}(k)} -  \vartheta(k-\ell)\right)\conj{k^2 \widehat{R}_2(k)}   i \omega(k) \widehat{R}_1(k-\ell) \ell^2 \widehat{R}_1(\ell)   \;d\ell dk.
\end{aligned}\]

And so these terms cancel except when $k$ or $\ell$ are small. Then when $k$ or $\ell$ are small, we can bound these two directly by Young's inequality and Cauchy-Schwartz. At first glance it would seem we should have lost a power of $\epsilon$ to the fact that $\widehat{\vartheta}^{-1}(k) \sim \mathcal{O}(\epsilon^{-1})$ for $k \approx 0;$ however, each term with a $\vartheta^{-1}$ also has a term of the form $i\omega(k)$ in Fourier space coming from the nonlinearity. This vanishes at $k=0$ and we therefore do not lose a power of $\epsilon.$ However, to get the necessary bound of $\epsilon^2(1+E) + \epsilon^3E^2$, we will assume that $\beta = 3$ henceforth. 

 We are left with six terms of $\mathcal{O}(\epsilon)$ from (\ref{epsterms}) and all the terms from $E_2.$ Before we begin the space-time approach it will be easiest to analyze the remaining terms in Fourier space. Since each of the terms are real, we can split them in two before using Plancharel's theorem to get

\begin{align}
 &2\epsilon \int   R_2 \cdot \vartheta^{-1} \Omega (\Psi_1 \vartheta R_1)   +\frac{1}{2} \Omega \Psi_2 \cdot  (\partial_x^2R_1)^2 +  \partial_x^2 R_2 \cdot \vartheta^{-1} \Omega (\partial_x^2\Psi_1 \vartheta R_1)\nonumber\\
& \hspace{.1 in}+ 2\partial_x^2 R_2 \cdot \vartheta^{-1} \Omega (\partial_x\Psi_1 \vartheta \partial_x R_1)  + \partial_x^2 R_2 \cdot \vartheta^{-1} \Omega (\Psi_1 \vartheta \partial_x^2R_1) - \partial_x^2 R_2 \cdot \Omega \left(\Psi_1 \cdot\partial_x^2 R_1 \right)   \;dx \nonumber \\
& = \frac{\sqrt{2}}{4} \epsilon \sum^2_{j,m,n=1} \iint \bigg(i\omega_j(k) \frac{\widehat{\vartheta}(\ell)}{\widehat{\vartheta}(k)} +  \frac{1}{2}  k^2 i\omega_m(k-\ell) \ell^2  + i\omega_j(k)k^2 (k-\ell)^2 \frac{\widehat{\vartheta}(\ell)}{\widehat{\vartheta}(k)}        \nonumber \\
&\hspace{.1 in}  +2 i\omega_j(k) k^2 (k-\ell) \ell    \frac{\widehat{\vartheta}(\ell)}{\widehat{\vartheta}(k)}  + i\omega_j(k)k^2 \ell^2    \frac{\widehat{\vartheta}(\ell)}{\widehat{\vartheta}(k)}     - i\omega_j(k)k^2 \ell^2 \bigg) \conj{\widehat{F}_j(k)} \widehat{G}_m(k-\ell)\widehat{F}_n(\ell) d\ell dk \label{epsfourier}
\end{align}
plus complex conjugate terms.


\section{The Space-Time Resonance Approach}

The Space-Time Resonance approach developed by Germain-Masmoudi-Shatah extends Shatah's normal form method \cite{Shatah} and, as pointed out in \cite{Germain1} and \cite{LannesST} is also related to Klainerman's vector-field method. In this paper, we will focus only on the time resonance half of the space-time resonance method.

 Consider the following simplified example
\[  \begin{aligned}
&\partial_t U= \Lambda U + \epsilon N(U,U) \\
&U(t) |_{t=0} = U_0
\end{aligned}\]
with $N(u,u)$ defined as a Fourier multiplier with some kernel function $m_j(k,k-\ell,\ell)$ as was done in (\ref{nonlinear}). We will similarly write $u$ in a rotating coordinate frame as $f = e^{-i\Lambda t}U$ and using Duhamel's formula we can write the Fourier transform of the solution $f$ as 
\begin{equation} \widehat{f}_j(t,k) = \widehat{U}_{j0}(k) + \epsilon \int_0^t \int e^{i\phi^j_{mn}(k,\ell) s}  m_j(k,k-\ell,\ell) \widehat{f}_m(k-\ell,s) \widehat{f}_n(\ell,s) d\ell ds  \label{duhamel} \end{equation}
with
 \[ \phi^j_{mn}(k,\ell) = -\omega_j(k) + \omega_m(k-\ell) + \omega_n(\ell) \]

The main idea of the space-time resonance approach is to then analyze the above using the method of stationary phase for both $s$ and $\ell.$ In the case of the variable $s,$ the phase is stationary when $\phi^j_{mn}(k,\ell)= 0.$ We will call that set the time resonances
\[  \mathcal{T} = \left\{(k,\ell)\; |   \phi^j_{mn}(k,\ell)= 0   \right\}. \] 
This is exactly the set of resonances found when applying Shatah's normal form method. If these resonances do not exist, or if we are integrating in some region bounded away from $\mathcal{T},$ then we can integrate (\ref{duhamel}) by parts with respect to $s$ and get 
\[\begin{aligned} 
 \widehat{f}_j(t,k) =& \widehat{U}_{j0}(k) + \epsilon  \int e^{i\phi^j_{mn}(k,\ell) s} \frac{m_j(k,k-\ell,\ell)}{i\phi^j_{mn}(k,\ell)} \widehat{f}_m(k-\ell,s) \widehat{f}_n(\ell,s) d\ell \bigg|_0^t \\
&- \epsilon \int_0^t \int e^{i\phi^j_{mn}(k,\ell) s}  \frac{m_j(k,k-\ell,\ell)}{i\phi^j_{mn}(k,\ell)} \partial_s\left( \widehat{f}_m(k-\ell,s) \widehat{f}_n(\ell,s) \right) d\ell ds.
\end{aligned}\]
The first integrand is small and no longer grows with respect to $t,$ the second is cubic in $f$ and $\mathcal{O}(\epsilon^2)$ when the time derivative is applied to either factor of $f.$

\subsection{Applying the Space-Time Resonance Approach}

In our case, we will need to avoid the set of time resonances, where \[ \phi^j_{mn}(k,\ell) = -\omega_j(k) + \omega_m(k-\ell) + \omega_n(\ell)=0 \]
 in the $k\ell$-plane. That is 
\[ \mathcal{T} = \left\{(k,\ell)\; | \; k=0, \; \ell = 0, \text{ or } k= \ell    \right\} . \] 

Fortunately, the region $k=0$ will be dealt with via a transparency condition for most terms and the addition of $E_2$ for the other. Also using the rescaling $\vartheta$ we have created a further transparency condition at $\ell=0.$ For $k=\ell$, recalling that
\[ \widehat{G}(k) = \widehat{G}^c(k) + \epsilon \widehat{G}^s(k),  \]
we can use the extra order of $\epsilon$ in order to bound $\partial_t E$ near $k=\ell$ without using any space-time methods to begin with.

In particular, without splitting up $G,$ each term in (\ref{epsfourier}) is bounded by $C\epsilon E.$ To see this we use the fact that $\omega(k)$ vanishes at zero to bound $\omega(k)/\widehat{\vartheta}(k)$ near $k=0$ and then Young's inequality and Cauchy-Schwartz. Therefore, if we split up $G,$ the terms with $G^s$ will be bounded by $C\epsilon^2E$ as a result of the extra order of $\epsilon$ in front of this factor. Thus we only need to use the space-time resonance approach in the region $\text{supp}(\widehat{G}^c(k-\ell)).$ This allows us to avoid the region where $k= \ell$ completely.

Since each of the aforementioned regions will be dealt with differently, we define the regions 
\[ \mathcal{W} = \left\{  (k,\ell) \in \text{supp}(\widehat{G}^c(k-\ell)) \; \big| \; |\ell|< \delta      \right\} ,  \]  
\[ \mathcal{Z} = \left\{  (k,\ell) \in \text{supp}(\widehat{G}^c(k-\ell)) \; \big| \; |k|< \delta      \right\},   \] 
and 
\[ \mathcal{V} = \text{supp}(\widehat{G}^c(k-\ell)) \backslash (\mathcal{W} \cup \mathcal{Z}) .\]

These regions can all be seen in the plot below.

\begin{figure} [h!] \label{fig:plot}
	\centering
		\includegraphics[scale=.25]{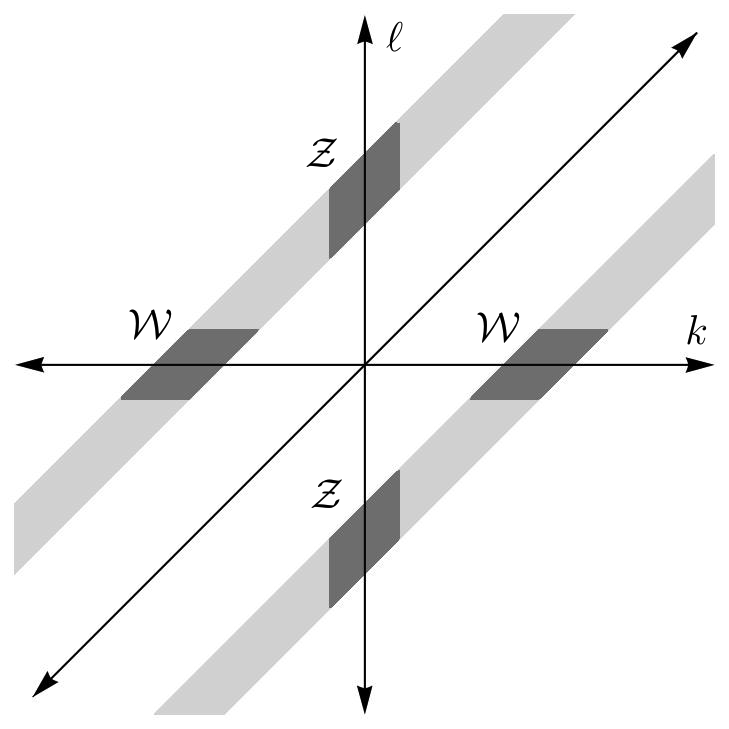}
		\caption{Partition of $k\ell$-plane.}
\end{figure}

We now bound each region separately.


\subsection{In the region $\mathcal{V}$}

Over the region $\mathcal{V}$ we have the advantage that $\widehat{\vartheta}^{-1}(k) = \widehat{\vartheta}(\ell)=1.$  Fortunately, this means that the last two terms from (\ref{epsfourier}), those with the most derivatives falling on $F,$ will both cancel.  We integrate the other four terms with respect to $t$ and adopt the notation used in (\ref{rotatingeqn}). This gives us equations of the form
\begin{equation}
\epsilon \sum_{j,m,n=1}^2 \int_0^t\iint_{\mathcal{V}} \widehat{a} ( k,k-\ell, \ell) e^{i \phi^j_{mn}(k,\ell)s}  \conj{\widehat{f}_j(k)} \widehat{g}^c_m(k-\ell) \widehat{f}_n(\ell) d\ell dk ds \label{STV}
\end{equation}
with the following kernel functions $\widehat{a}(k,k-\ell,\ell)$ coming from (\ref{epsfourier}) 
\[ i\omega_j(k), \; \frac{1}{2}  k^2 i\omega_m(k-\ell) \ell^2,\; i\omega_j(k)k^2 (k-\ell)^2 ,\text{ and } \; 2 i\omega_j(k) k^2 (k-\ell) \ell . \]
We note that there are similar complex conjugate terms that can be dealt with in the same manner.
 We integrate by parts with respect to $s$ giving us
\begin{align}
&\epsilon \sum_{j,m,n=1}^2  \iint_{\mathcal{V}}  \frac{\widehat{a} ( k,k-\ell, \ell)}{i \phi^j_{mn}(k,\ell)}e^{i \phi^j_{mn}(k,\ell)s}  \conj{\widehat{f}_j(k)} \widehat{g}^c_m(k-\ell) \widehat{f}_n(\ell) d\ell dk \bigg|_0^t\label{boundary1} \\  
&-\epsilon \sum_{j,m,n=1}^2  \int_0^t\iint_{\mathcal{V}}  \frac{\widehat{a} ( k,k-\ell, \ell)}{i \phi^j_{mn}(k,\ell)}e^{i \phi^j_{mn}(k,\ell)s}  \partial_s\left(\conj{\widehat{f}_j(k)} \widehat{g}^c_m(k-\ell) \widehat{f}_n(\ell)\right) d\ell dk ds. \label{IBPt}
\end{align}

Since $\mathcal{T} \cap \mathcal{V}  = \emptyset,$ the boundary term (\ref{boundary1}) is well-defined and bounded by $C\epsilon E.$ Thus we can subtract it from the left-hand side of our estimate. In essence, by making this subtraction we are modifying the energy by this boundary term. However, since it is bounded by $C\epsilon E,$ this does not change the energy in a significant way. 

For (\ref{IBPt}), we will focus on one term. The rest are similar. Ignoring the summation for the moment, we have

\begin{align}
\epsilon &  \int_0^t\iint_{\mathcal{V} }  \frac{\widehat{a} ( k,k-\ell, \ell)}{i \phi^j_{mn}(k,\ell)}e^{i \phi^j_{mn}(k,\ell)s}  \partial_s\left(\conj{\widehat{f}_j(k)}\right) \widehat{g}^c_m(k-\ell) \widehat{f}_n(\ell) d\ell dk ds \nonumber \\
& = 2\epsilon^2   \int_0^t\iint_{\mathcal{V} }  \frac{\widehat{a} ( k,k-\ell, \ell)}{i \phi^j_{mn}(k,\ell)}e^{i (\omega_m(k-\ell) + \omega_n(\ell))s}  \conj{\frac{1}{\widehat{\vartheta}(k)}\widehat{N}_j(e^{i\Lambda s}\widehat{g},e^{i\Lambda s}\widehat{\vartheta}\widehat{f})(k)} \widehat{g}^c_m(k-\ell) \widehat{f}_n(\ell) d\ell dk ds \nonumber\\
& + \epsilon^{\beta+1}   \int_0^t\iint_{\mathcal{V} }  \frac{\widehat{a} ( k,k-\ell, \ell)}{i \phi^j_{mn}(k,\ell)}e^{i  (\omega_m(k-\ell) + \omega_n(\ell))s}  \conj{\frac{1}{\widehat{\vartheta}(k)}\widehat{N}_j(e^{i\Lambda s}\widehat{\vartheta}\widehat{f},e^{i\Lambda s}\widehat{\vartheta}\widehat{f})(k)} \widehat{g}^c_m(k-\ell) \widehat{f}_n(\ell) d\ell dk ds \nonumber\\
& + \epsilon^{1-\beta}   \int_0^t\iint_{\mathcal{V} }  \frac{\widehat{a} ( k,k-\ell, \ell)}{i \phi^j_{mn}(k,\ell)}e^{i  (\omega_m(k-\ell) + \omega_n(\ell))s} \conj{e^{i\omega_j(k)s}\frac{1}{\widehat{\vartheta}(k)} \text{Res}(\epsilon \widehat{G})} \widehat{g}^c_m(k-\ell) \widehat{f}_n(\ell) d\ell dk ds  \nonumber\\
& = 2\epsilon^2   \int_0^t\iint_{\mathcal{V} }  \frac{\widehat{a} ( k,k-\ell, \ell)}{i \phi^j_{mn}(k,\ell)} \conj{\widehat{N}_j(\widehat{G},\widehat{\vartheta}\widehat{F})(k)} \widehat{G}_m^c(k-\ell) \widehat{F}_n(\ell) d\ell dk ds \label{gfnon} \\
& + \epsilon^{\beta+1}  \int_0^t\iint_{\mathcal{V} }  \frac{\widehat{a} ( k,k-\ell, \ell)}{i \phi^j_{mn}(k,\ell)} \conj{\widehat{N}_j(\widehat{\vartheta}\widehat{F},\widehat{\vartheta}\widehat{F})(k)} \widehat{G}_m^c(k-\ell) \widehat{F}_n(\ell) d\ell dk ds \label{ffnon}\\
& + \epsilon^{1-\beta}  \int_0^t\iint_{\mathcal{V} }  \frac{\widehat{a} ( k,k-\ell, \ell)}{i \phi^j_{mn}(k,\ell)} \conj{\text{Res}(\epsilon \widehat{G})} \widehat{G}_m^c(k-\ell) \widehat{F}_n(\ell) d\ell dk ds \label{resnon}
\end{align}
where we dropped the $1/\widehat{\vartheta}(k)$ in the second equality because the integral is over $\mathcal{V}$ where this function is equal to $1.$
The residual term (\ref{resnon}) can be bounded directly by $C\epsilon^2 E^{1/2}$ using Young's inequality and Cauchy-Schwartz. Furthermore, for (\ref{gfnon}) and (\ref{ffnon}), all of the kernel functions, with the exception of \[\frac{\sqrt{2}}{8}  k^2 i\omega_m(k-\ell) \ell^2 \] can be bounded using integration by parts, which in this context is really writing $k = (k-\ell) + \ell,$ as well as Cauchy-Schwartz and Young's Inequality. For example

\[\begin{aligned}
&  \iint_{\mathcal{V} }  \frac{2i\omega_j(k)k^2(k-\ell)\ell}{i \phi^j_{mn}(k,\ell)} \conj{\widehat{N}_j(\widehat{\vartheta}\widehat{F},\widehat{\vartheta}\widehat{F})(k)} \widehat{G}_m^c(k-\ell) \widehat{F}_n(\ell) d\ell dk  \\
&= - \iiint_{\mathbb{R} \times \mathcal{V} }  \frac{\sqrt{2}i\omega_j(k) \omega_j(k)k^2(k-\ell)\ell}{i \phi^j_{mn}(k,\ell)} \conj{\widehat{\vartheta}(k-p)\widehat{F}(k-p)\widehat{\vartheta}(p)\widehat{F}(p)} \widehat{G}_m^c(k-\ell) \widehat{F}_n(\ell) dp d\ell dk  \\
& \leq C \int \left(\int \left|k^2 \widehat{F}(k-p)\widehat{F}(p)\right|dp \right)\left(  \int\left|k(k-\ell) \widehat{G}_m^c(k-\ell) \ell \widehat{F}_n(\ell)\right| d\ell\right) dk \\
& \leq C \int \left(\int \left|(1+(k-p)^2 +p^2) \widehat{F}(k-p)\widehat{F}(p)\right|dp \right)  \\
&\hspace{2 in} \times \left(  \int \left|(1+(k-\ell)+ \ell)(k-\ell) \widehat{G}_m^c(k-\ell) \ell \widehat{F}_n(\ell)\right| d\ell\right) dk \\
& \leq C \|\widehat{F}\|_{L^1}\|F\|^2_{H^2} \\
& \leq C \|F\|_{H^2}^3 \lesssim E^{3/2}.
\end{aligned}\]

The second to last step used Cauchy-Schwarz, Young's inequality, and the fact that $\|\widehat{G}^c\|_{L^1} \leq C$ independent of $\epsilon.$ The last step then used \[\|\widehat{F}\|_{L^1} \leq C \|(1+(\cdot)\widehat{F}(\cdot)\|_{L^2} \leq CE^{1/2}.\]

The final term with kernel function $\sqrt{2}/8  k^2 i\omega_m(k-\ell) \ell^2 $ cannot be bounded this way due to the high powers of $k$ and $\ell.$ Since each power of $k$ or $\ell$ is essentially one derivative in real-space, what we have is two derivatives falling on each power of $F.$ When we consider the nonlinear term acting on $F$ as well this puts $5/2$-derivatives on one of the factors of $F.$ 

Therefore, for this last one we must take advantage of the form of (\ref{approx}) and integrate by parts in a different way. First, we simplify the terms in (\ref{gfnon}) with this kernel to
\[\begin{aligned} & -\frac{1}{4} \epsilon^2 \sum_{j=1}^2 \int_0^t  \sum_{m,n=1}^2 \iint_{\mathcal{V} }  \frac{i\omega_j(k) \cdot \omega_m(k-\ell)}{\phi^j_{mn}(k,\ell)} \widehat{G}_m^c(k-\ell) \ell^2 \widehat{F}_n(\ell)    \\
&\hspace{2 in}\times \conj{\left( \sum_{q,r=1}^2 \int k^2 \widehat{G}_q(k-p)\widehat{\vartheta}(p)\widehat{F}_r(p) dp   \right)} d\ell dkds.
\end{aligned}\]

 For $j \neq n,$ we can approximate the kernel function using the fact that $k-\ell \approx \pm k_0.$ This gives us \[  \left| \frac{i\omega_j(k) \cdot \omega_m(k-\ell)}{\phi^j_{mn}(k,\ell)}  - i\omega_m(k-\ell) \right| \leq C \] for all $(k,\ell) \in \mathcal{V} .$ 
Using this estimate on the kernel function, the integrand can then be estimated directly using the method shown for the other kernel functions. When $j = n,$ we approximate the kernel again with
 \[  \left| \frac{i\omega_j(k) \cdot \omega_m(k-\ell)}{\phi^j_{mn}(k,\ell)}   + i\omega_j(k) \right| \leq C \] for all $(k,\ell) \in \mathcal{V} .$  Unfortunately, in this case there still appears to be an extra half-derivative and so this cannot be estimated directly. We proceed in a fashion very similar to \cite{Hunter}. We can drop $\widehat{\vartheta} (p)$ since it only effects this term with $p$ small, in which case the error can be estimated directly. We can also integrate over all of $\mathbb{R}^2$ instead of $\mathcal{V}$ for the same reason. Now using Plancharel's Theorem and the fact that $F$ and $G$ are real, we have
\[\begin{aligned}  \frac{1}{4} \epsilon^2 & \int_0^t \int \left( \sum_{j,m=1}^2 \int i\omega_j(k)  \widehat{G}_m^c(k-\ell) \ell^2 \widehat{F}_j(\ell) d\ell\right)    \conj{\left( \sum_{q,r=1}^2 \int k^2 \widehat{G}_q(k-p)\widehat{F}_r(p) dp   \right)} dkds\\
=& \frac{1}{4} \epsilon^2  \int_0^t \int \left( \sum_{j,m=1}^2 \Omega_j\left(    G_m^c \partial_x^2 F_j \right)\right)   \left( \partial_x^2\left(\sum_{q,r=1}^2 G_q F_r \right) \right) \;dxds\\
=& \frac{1}{4} \epsilon^2  \int_0^t \int \left(\Omega_1\left(  \left(G_1^c + G_2^c\right) \partial_x^2 F_1 \right)  + \Omega_2\left(  \left(G_1^c + G_2^c\right) \partial_x^2 F_2 \right)\right)   \partial_x^2\left(\left(G_1 + G_2 \right) \left(  F_1 +  F_2  \right)\right) dxds\\
=& \frac{1}{4} \epsilon^2  \int_0^t \int \Omega\left( \partial_x^2  \left(F_1-F_2 \right) \left(G_1^c + G_2^c\right) \right)\partial_x^2\left(\left(G_1 + G_2 \right) \left(  F_1 +  F_2  \right)\right)  \;dxds\\
=& \epsilon^2  \int_0^t \int \Omega\left( \partial_x^2  R_2  \cdot \Psi_1^c\right)   \partial_x^2\left(\Psi_1 \cdot R_1 \right)  \;dxds\\
=&\epsilon^2  \int_0^t \int \Omega\left(\partial_x^2  R_2 \right) \cdot \Psi_1^c \cdot \Psi_1 \cdot  \partial_x^2  R_1   \;dxds \\
&+ \epsilon^2  \int_0^t \int \left[\Omega\;,\; \Psi_1^c\right] \partial_x^2  R_2 \cdot \Psi_1 \cdot \partial_x^2 R_1   \;dxds \\
&+ \epsilon^2  \int_0^t \int \Omega\left(\partial_x^2  R_2 \cdot \Psi_1^c \right)  \left[ \partial_x^2\;,\; \Psi_1 \right] R_1   \;dxds\\
\end{aligned}\]
The last two commutators can be estimated directly using Corollary \ref{prop:commcorr} and Proposition \ref{prop:comm2} with Lemma \ref{lem:approx} from the Appendix. Then 
\[\begin{aligned}
 \epsilon^2&  \int_0^t \int \Omega\left(\partial_x^2  R_2 \right) \Psi_1^c \cdot \Psi_1 \cdot  \partial_x^2  R_1 \;dxds   \\
=&  \epsilon^2  \int_0^t \int \partial_x^2  \left(\partial_t R_1 -\epsilon^{-\beta}\vartheta^{-1}\text{Res}(\epsilon \Psi) \right) \Psi_1^c\cdot \Psi_1 \cdot \partial_x^2  R_1  \;dxds  \\
=& \frac{1}{2}\epsilon^2  \int \left(\partial_x^2  R_1\right)^2 \left( \Psi_1^c\Psi_1 \right)\;dx \bigg|_0^t  -\frac{1}{2}\epsilon^2  \int_0^t \int \left(\partial_x^2  R_1\right)^2 \partial_t\left( \Psi_1^c\Psi_1 \right)\;dxds \\
& -\epsilon^{2-\beta} \int_0^t \int \partial_x^2 \vartheta^{-1} \text{Res}(\epsilon\Psi) \cdot \Psi_1^c\cdot \Psi_1 \cdot \partial_x^2 R_1  \;dxds
\end{aligned}\]
Now the first term can be treated as earlier boundary terms; we subtract it from the left-hand side of our estimate using the fact that it is both small and bounded independent of $t.$ Then both of the other terms can be bounded by $C\epsilon^2(1+E)$ directly. The integrand (\ref{ffnon}) can be bounded using the same method used for (\ref{gfnon}). For the most part this has $G_q$ replaced with $F_q$ and the real difference between these two terms comes at the point of 
\[\begin{aligned}
  \epsilon^{\beta+1}& \int_0^t \int \Omega\left( \partial_x^2  R_2  \cdot \Psi_1^c\right)   \partial_x^2\left(R_1 \cdot R_1 \right)  \;dxds\\
=&2\epsilon^{\beta+1} \int_0^t \int \Omega\left( \partial_x^2  R_2  \cdot \Psi_1^c\right)   (\partial_x^2R_1 \cdot R_1 + (\partial_x R_1)^2)  \;dxds.
\end{aligned}\]
The second term here can be bounded using the following
\[\begin{aligned}
2\epsilon^{\beta+1} &\int_0^t \int \Omega\left( \partial_x^2  R_2  \cdot \Psi_1^c\right)    (\partial_x R_1)^2  \;dxds\\
=&  -2\epsilon^{\beta+1} \int_0^t \int \left( \partial_x^2  R_2  \cdot \Psi_1^c\right)   \Omega\left( (\partial_x R_1)^2\right) \;dxds\\
\leq & C\epsilon^{\beta+1} \int_0^t \|R_2\|_{H^2} \|(\partial_xR_1)^2\|_{H^{1}}\;ds \\
\leq & C\epsilon^{\beta+1} \int_0^t \|R_2\|_{H^2} \|R_1\|^2_{H^{2}}\;ds \\
\leq &  C\epsilon^{\beta+1} \int_0^t E^{3/2}\;ds
\end{aligned}\]

The first term can be bounded using similar commutator estimates from (\ref{gfnon})
\[\begin{aligned}
&2\epsilon^{\beta+1} \int_0^t \int \Omega\left( \partial_x^2  R_2  \cdot \Psi_1^c\right)   \partial_x^2R_1 \cdot R_1    \;dxds\\
&=2\epsilon^{\beta+1} \int_0^t \int \Omega\left( \partial_x^2  R_2 \right)  \cdot \Psi_1^c   \cdot\partial_x^2R_1 \cdot R_1    dxds +2\epsilon^{\beta+1} \int_0^t \int \left[\Omega \;,\; \Psi_1^c\right] \partial_x^2  R_2 \cdot   \partial_x^2R_1 \cdot R_1    dxds\\
&=2\epsilon^{\beta+1} \int_0^t \int \partial_x^2 \left( \partial_t  R_1 -\epsilon^{-\beta}\vartheta^{-1}\text{Res}(\epsilon \Psi)\right)  \cdot \Psi_1^c   \cdot\partial_x^2R_1 \cdot R_1    \;dxds\\
&\hspace{1 in} +2\epsilon^{\beta+1} \int_0^t \int \left[\Omega \;,\; \Psi_1^c\right] \partial_x^2  R_2 \cdot   \partial_x^2R_1 \cdot R_1    \;dxds\\
&=\epsilon^{\beta+1} \int \partial_x^2  R_1 ^2\cdot  \Psi_1^c   \cdot R_1    \;dx \bigg|_0^t-\epsilon^{\beta+1} \int_0^t \int \partial_x^2  R_1 ^2\cdot  \partial_t\left( \Psi_1^c   \cdot R_1 \right)   \;dxds \\
&-2\epsilon\int_0^t\int \vartheta^{-1}\text{Res}(\epsilon \Psi)  \cdot \Psi_1^c   \cdot\partial_x^2R_1 \cdot R_1    dxds
 +2\epsilon^{\beta+1} \int_0^t \int \left[\Omega \;,\; \Psi_1^c\right] \partial_x^2  R_2 \cdot   \partial_x^2R_1 \cdot R_1    dxds\\
\end{aligned}\]
Again we subtract the boundary term from the left-hand side, the commutator is estimated using Corollary \ref{prop:commcorr}, and the other two terms can be estimated using Cauchy-Schwartz.

 Thus, over the region $\mathcal{V}$ we have \[ \int_0^t \partial_s E(s) \; ds \lesssim   \int_0^t \epsilon^2(1+E) + \epsilon^3E^2\; ds. \]


\subsection{In the region with $\ell$ small, $\mathcal{W}$}

Again we are hoping to bound the terms from (\ref{epsfourier}) but now over the region $\mathcal{W}.$ Here we use the fact that $\widehat{\vartheta}^{-1}(k) = 1$ in $\mathcal{W}.$ This leaves us with two groups of terms, those with $\widehat{\vartheta}(\ell)$ and those without. For the terms with $\widehat{\vartheta}(\ell),$ define $\vartheta_0 = \vartheta - \epsilon.$ This splits the integral up into two. One term will now have $\vartheta_0(\ell) = 0$ when $\ell=0.$ The other is of order $\epsilon^2$ and so can be estimated directly without any added integration by parts. For the first term we integrate with respect to $t$ and move to the rotating coordinate system to get 
\[
\epsilon \sum_{j,m,n=1}^2 \int_0^t\iint_{\mathcal{W}} \widehat{a} ( k,k-\ell, \ell) e^{i \phi^j_{mn}(k,\ell)s}  \conj{\widehat{f}_j(k)} \widehat{g}_m(k-\ell) \widehat{\vartheta}_0(\ell)\widehat{f}_n(\ell) d\ell dk ds
\]
Then when we integrate by parts with respect to $s,$ we get 
\begin{align}
&\epsilon \sum_{j,m,n=1}^2  \iint_{ \mathcal{W}}  \frac{\widehat{a} ( k,k-\ell, \ell)}{i \phi^j_{mn}(k,\ell)}e^{i \phi^j_{mn}(k,\ell)s}   \conj{\widehat{f}_j(k)} \widehat{g}^c_m(k-\ell) \widehat{\vartheta}_0(\ell)\widehat{f}_n(\ell) d\ell dk \bigg|_0^t\label{boundary1W} \\  
&-\epsilon \sum_{j,m,n=1}^2  \int_0^t\iint_{ \mathcal{W}}  \frac{\widehat{a} ( k,k-\ell, \ell)}{i \phi^j_{mn}(k,\ell)}e^{i \phi^j_{mn}(k,\ell)s} \widehat{\vartheta}_0(\ell) \partial_s\left(\conj{\widehat{f}_j(k)} \widehat{g}^c_m(k-\ell) \widehat{f}_n(\ell)\right) d\ell dk ds. \label{IBPtW}
\end{align}

We note that $\mathcal{T} \cap \mathcal{W} = \{\ell=0\}.$ However, $\widehat{\vartheta}_0(0) = 0$ as well and so for $(k,\ell) \in \mathcal{W},$ \[
\left| \frac{\widehat{a} ( k,k-\ell, \ell)\widehat{\vartheta}_0(\ell)}{i \phi^j_{mn}(k,\ell)}\right| < C.
\] Therefore, the boundary term (\ref{boundary1W}) is well-defined and bounded by $\epsilon E.$ We can subtract it to the left-hand side of our estimate as we did with the boundary term in $\mathcal{V}.$

For the non-boundary term we can use the fact that $\mathcal{W}$ is compact to bound all the terms directly using Young's inequality and Cauchy-Schwartz. In particular, since all the kernels can be bounded by a constant, we do not need any of the extra integration by parts techniques that we used in $\mathcal{V}.$

Finally, for the two terms without $\widehat{\vartheta}(\ell)$ we note that both terms do not need an added transparency condition because they have a factor $\ell^2$ in the kernel already. This will give us 
\[
\left| \frac{\widehat{a} ( k,k-\ell, \ell)}{i \phi^j_{mn}(k,\ell)}\right| < C|\ell|
\] 
and so we can bound these terms directly after integrating by parts as well.


\subsection{In the region with $k$ small, $\mathcal{Z}$}

Over the region $\mathcal{Z}$ we have both the terms from (\ref{epsfourier}) as well as those from $E_2.$ Due to the fact that $\widehat{\vartheta}^{-1}(k)$ is $\mathcal{O}(\epsilon^{-1})$ near $k=0$ we will need to show that terms of formal order as high as $\mathcal{O}(\epsilon^3)$ are bounded by $ \epsilon^2(1+E) + \epsilon^3E^2.$ As we will see the terms in $E_2$ were chosen precisely because they give cancellations of terms that are not of $\mathcal{O}(\epsilon^2E)$ and to which we can't apply the method of space-time resonances in what follows. We note that  $\widehat{\vartheta}(\ell) =1$ in $\mathcal{Z}.$

For $\mathcal{O}(\epsilon)$ we have 
\begin{align}
&\frac{\sqrt{2}}{4} \epsilon \sum^2_{j,m,n=1} \iint_{\mathcal{Z}} \bigg(i\omega_j(k) \frac{1}{\widehat{\vartheta}(k)} +  \frac{1}{2}  k^2 i\omega_m(k-\ell) \ell^2  + i\omega_j(k)k^2 (k-\ell)^2 \frac{1}{\widehat{\vartheta}(k)}        \nonumber \\
&\hspace{.1 in}  +2 i\omega_j(k) k^2 (k-\ell) \ell    \frac{1}{\widehat{\vartheta}(k)}  + i\omega_j(k)k^2 \ell^2    \frac{1}{\widehat{\vartheta}(k)}     - i\omega_j(k)k^2 \ell^2 \bigg) \conj{\widehat{F}_j(k)} \widehat{G}_m^c(k-\ell)\widehat{F}_n(\ell) \;d\ell dk \nonumber\\
&+ \frac{1}{2}\sum_{j=1}^2  \epsilon \int\displaylimits_{|k| < \delta} 
\conj{\Omega_j \widehat{F}_j} \cdot B_j(\widehat{G}^c,\widehat{F})    +    \conj{\widehat{F}_j}\cdot B_j( \widehat{G}^c, \Lambda \widehat{F}) + \conj{\widehat{F}_j} \cdot B_j(\Lambda \widehat{G}^c, \widehat{F}) \nonumber\\
& \hspace{1 in} +  \conj{\widehat{F}_j}\cdot  B_j((\partial_t-\Lambda) \widehat{G}^c, \widehat{F}) \;dk   \nonumber\\
&= \frac{\sqrt{2}}{4} \epsilon \sum^2_{j,m,n=1} \iint_{\mathcal{Z}} \bigg(  \frac{1}{2}  k^2 i\omega_m(k-\ell) \ell^2  + i\omega_j(k)k^2 (k-\ell)^2 \frac{1}{\widehat{\vartheta}(k)}    +2 i\omega_j(k) k^2 (k-\ell) \ell    \frac{1}{\widehat{\vartheta}(k)}     \nonumber \\
&\hspace{1 in}   + i\omega_j(k)k^2 \ell^2    \frac{1}{\widehat{\vartheta}(k)}     - i\omega_j(k)k^2 \ell^2 \bigg) \conj{\widehat{F}_j(k)} \widehat{G}_m^c(k-\ell)\widehat{F}_n(\ell) \;d\ell dk \label{zstres}\\
&+ \frac{1}{2}\sum_{j=1}^2  \epsilon \int\displaylimits_{|k| < \delta}   \conj{\widehat{F}_j} \bigg( \frac{1}{\widehat{\vartheta} (k) } N_j(\widehat{G}^c,\widehat{\vartheta}\widehat{F})    -\Omega_j  B_j( \widehat{G}^c,  \widehat{F})+   B_j( \widehat{G}^c, \Lambda \widehat{F}) +  B_j(\Lambda \widehat{G}^c, \widehat{F})\bigg)  \;dk \label{canceleps2} \\
&+ \frac{1}{2}\sum_{j=1}^2  \epsilon \int\displaylimits_{|k| < \delta}  \conj{\widehat{F}_j}\cdot  B_j((\partial_t-\Lambda) \widehat{G}^c, \widehat{F}) \;dk \label{gevo}
\end{align}
with similar complex conjugate terms. We will use space-time resonance methods on (\ref{zstres}); the form of $E_2$ was chosen precisely to ensure that the terms in (\ref{canceleps2}) all cancel; and as shown in \cite{CEW}, $\|(\partial_t-\Lambda) \widehat{G}^c\|_{L^2} \leq C \epsilon^2,$ and so the last term (\ref{gevo}) can be bounded directly.

For (\ref{zstres}) we first integrate with respect to $t$ and move to a rotating coordinate frame giving us three terms of the form 
\[   
\epsilon \sum_{j,m,n=1}^2 \int_0^t\iint_{\mathcal{Z}} \widehat{a} ( k,k-\ell, \ell) e^{i \phi^j_{mn}(k,\ell)s} \widehat{\vartheta}^{-1}(k) \conj{\widehat{f}_j(k)} \widehat{g}^c_m(k-\ell) \widehat{f}_n(\ell) d\ell dk ds
\]
as well as two others without a factor of $\widehat{\vartheta}^{-1}(k) .$ We integrate by parts with respect to $s.$ Those terms which contain a factor of $\widehat{\vartheta}^{-1}(k)$ are 
\begin{align}
&\epsilon \sum_{j,m,n=1}^2  \iint_{ \mathcal{Z}}  \frac{\widehat{a} ( k,k-\ell, \ell)}{i \phi^j_{mn}(k,\ell)}e^{i \phi^j_{mn}(k,\ell)s}   \widehat{\vartheta}^{-1}(k)\conj{\widehat{f}_j(k)} \widehat{g}^c_m(k-\ell) \widehat{f}_n(\ell) d\ell dk \bigg|_0^t\label{boundary1Z} \\  
&-\epsilon \sum_{j,m,n=1}^2  \int_0^t\iint_{ \mathcal{Z}}  \frac{\widehat{a} ( k,k-\ell, \ell)}{i \phi^j_{mn}(k,\ell)}e^{i \phi^j_{mn}(k,\ell)s} \widehat{\vartheta}^{-1}(k) \partial_s\left(\conj{\widehat{f}_j(k)} \widehat{g}^c_m(k-\ell) \widehat{f}_n(\ell)\right) d\ell dk ds. \label{IBPtZ}
\end{align}

We note that $\mathcal{T} \cap \mathcal{Z} = \{k=0\}.$ However, $\widehat{a}(0,-\ell,\ell) = 0$ and it approaches $0$ cubically. So for $(k,\ell) \in \mathcal{Z},$ \[
\left| \frac{\widehat{a} ( k,k-\ell, \ell)}{i \phi^j_{mn}(k,\ell)}\right| < C|k|^2.
\] Therefore, both terms are well-defined, and moreover, each of the kernel functions approach zero so rapidly we have 
\[
\left| \frac{\widehat{a} ( k,k-\ell, \ell) \widehat{\vartheta}^{-1}(k) }{i \phi^j_{mn}(k,\ell)}\right| < C|k|.
\] 
Therefore we use the fact that $\mathcal{Z}$ is compact to bound all the terms directly as we did in $\mathcal{W}.$ For those last terms which don't contain a factor of $\vartheta^{-1}(k)$ we note that 
\[
\left| \frac{\widehat{a} ( k,k-\ell, \ell)}{i \phi^j_{mn}(k,\ell)}\right| < C|k|
\]
 for $(k,\ell) \in \mathcal{Z}.$ Thus these terms can still be bounded directly once we integrate by parts.


For the $\mathcal{O}(\epsilon^2)$ terms in $\mathcal{Z}$ we again note that we can restrict to $\widehat{G}^c$ and that $\|(\partial_t-\Lambda) \widehat{G}^c\|_{L^2} \leq C \epsilon^2.$ This leaves us with

\[\begin{aligned}
 \epsilon^2\sum_{j=1}^2  &\int\displaylimits_{|k| < \delta}  \conj{\left( \widehat{\vartheta}^{-1} \widehat{N}_j(\widehat{G}^c,\widehat{\vartheta}\widehat{F})  +  B_j(\Lambda \widehat{G}^c, \widehat{F})  +B_j( \widehat{G}^c, \Lambda\widehat{F})    \right)} \cdot B_j(\widehat{G}^c,\widehat{F})  \\
&\hspace{.5 in}+ \conj{\widehat{F}_j} \cdot B_j(\widehat{G}^c,  \widehat{\vartheta}^{-1} \widehat{N}(\widehat{G}, \vartheta \widehat{F}) )\;dk +c.c.
\end{aligned}\]

The first term and its complex conjugate can be written

\[\begin{aligned}
 &\epsilon^2  \sum_{j=1}^2\int\displaylimits_{|k| < \delta}  \conj{ \widehat{\vartheta}^{-1} \widehat{N}(\widehat{G}^c, \widehat{\vartheta} \widehat{F})    } \cdot B_j(\widehat{G}^c,\widehat{F})  +   \widehat{\vartheta}^{-1} \widehat{N}(\widehat{G}^c, \widehat{\vartheta} \widehat{F})     \cdot \conj{B_j(\widehat{G}^c,\widehat{F})  } \;dk\\
& = \epsilon^2  \sum_{\substack{j,m,n,\\ q,r=1}}^2\iiint\displaylimits_{|k| < \delta}  \conj{  \frac{i\omega_j(k)}{\sqrt{2}} \frac{\widehat{\vartheta}(p)}{\widehat{\vartheta}(k)}  \widehat{G}_m^c(k-p) \widehat{F}_n(p)      } \cdot \frac{\omega_j(k)/\sqrt{2}}{\phi^j_{qr}(k,\ell)} \frac{\widehat{\vartheta}(\ell)}{\widehat{\vartheta}(k)} \widehat{G}^c_q(k-\ell) \widehat{F}_r(\ell)    \\
& \hspace{1 in} +     \frac{i\omega_j(k)}{\sqrt{2}} \frac{\widehat{\vartheta}(\ell)}{\widehat{\vartheta}(k)}  \widehat{G}^c_q(k-\ell) \widehat{F}_r(\ell)       \cdot \conj{\frac{\omega_j(k)/\sqrt{2}}{\phi^j_{mn}(k,p)} \frac{\vartheta(p)}{\widehat{\vartheta}(k)} \widehat{G}_m^c(k-p) \widehat{F}_n(p) }   \;dp d\ell dk\\
& = \frac{1}{2}\epsilon^2  \sum_{\substack{j,m,n,\\ q,r=1}}^2\iiint\displaylimits_{|k| < \delta}    i\omega_j(k)\omega_j(k) \frac{\widehat{\vartheta}(p)\widehat{\vartheta}(\ell)}{\widehat{\vartheta}(k)\widehat{\vartheta}(k)}       \left( \frac{-\omega_q(k-\ell) - \omega_r(\ell) + \omega_m(k-p) + \omega_n(p)}{\phi^j_{mn}(k,p)\phi^j_{qr}(k,\ell)} \right)\\
&\hspace{2 in } \times \conj{\widehat{G}_m^c(k-p) \widehat{F}_n(p)} \widehat{G}^c_q(k-\ell) \widehat{F}_r(\ell)   dpd\ell dk \\
\end{aligned}\]

The second and third term as well as their complex conjugates are

\[\begin{aligned}
 &\epsilon^2  \sum_{j=1}^2 \int\displaylimits_{|k| < \delta}  \conj{\big(  B_j(\Lambda \widehat{G}^c, \widehat{F})   + B_j( \widehat{G}^c, \Lambda\widehat{F})   \big)}  B_j(\widehat{G}^c,\widehat{F})  +\big(   B_j(\Lambda \widehat{G}^c, \widehat{F})  +B_j( \widehat{G}^c, \Lambda\widehat{F})  \big)  \conj{B_j(\widehat{G}^c,\widehat{F}) } dk \\
 =& \epsilon^2 \hspace{-.05 in} \sum_{\substack{j,m,n,\\ q,r=1}}^2\iiint\displaylimits_{|k| < \delta}  \conj{  \frac{\omega_j(k)/\sqrt{2}}{\phi^j_{mn}(k,p)} \frac{\widehat{\vartheta}(p)}{\widehat{\vartheta}(k)} ( i\omega_m(k-p) + i\omega_n(p)  )\widehat{G}_m^c(k-p) \widehat{F}_n(p)      }  \frac{\omega_j(k)/\sqrt{2}}{\phi^j_{qr}(k,\ell)} \frac{\widehat{\vartheta}(\ell)}{\widehat{\vartheta}(k)} \widehat{G}^c_q(k-\ell) \widehat{F}_r(\ell)    \\
&+   \frac{\omega_j(k)/\sqrt{2}}{\phi^j_{qr}(k,\ell)} \frac{\widehat{\vartheta}(\ell)}{\widehat{\vartheta}(k)} \left( i\omega_q(k-\ell) + i\omega_r(\ell)  \right)\widehat{G}^c_q(k-\ell) \widehat{F}_r(\ell)       \cdot \conj{ \frac{\omega_j(k)/\sqrt{2}}{\phi^j_{mn}(k,p)} \frac{\widehat{\vartheta}(p)}{\widehat{\vartheta}(k)} \widehat{G}_m^c(k-p) \widehat{F}_n(p) }d\ell dpdk   \\
 =&-\frac{1}{2} \epsilon^2  \sum_{\substack{j,m,n,\\ q,r=1}}^2\iiint\displaylimits_{|k| < \delta}    i\omega_j(k)\omega_j(k) \frac{\widehat{\vartheta}(p)\widehat{\vartheta}(\ell)}{\widehat{\vartheta}(k)\widehat{\vartheta}(k)}       \left( \frac{-\omega_q(k-\ell) - \omega_r(\ell) + \omega_m(k-p) + \omega_n(p)}{\phi^j_{mn}(k,p)\phi^j_{qr}(k,\ell)} \right)\\
&\hspace{2 in } \times \conj{\widehat{G}_m^c(k-p) \widehat{F}_n(p)} \widehat{G}^c_q(k-\ell) \widehat{F}_r(\ell) d\ell dpdk.    
\end{aligned}\]

These both cancel. The last term is of the form

\[\begin{aligned}
 &\epsilon^2 \sum_{j=1}^2  \int\displaylimits_{|k| < \delta}   \conj{\widehat{F}_j} \cdot B_j(\widehat{G}^c,  \widehat{\vartheta}^{-1} N(\widehat{G}, \widehat{\vartheta} \widehat{F}) )dk \\
&=  \epsilon^2  \sum_{\substack{j,m,n,\\ q,r=1}}^2\iiint\displaylimits_{|k| < \delta}   \conj{\widehat{F}_j(k)} \cdot \frac{\omega_j(k)/\sqrt{2}}{\phi^j_{mn}(k,\ell)} \frac{\widehat{\vartheta}(\ell)}{\widehat{\vartheta}(k)}  \widehat{G}_m^c(k-\ell)  \frac{\widehat{\vartheta}(p)}{\widehat{\vartheta}(\ell)} \frac{i\omega_n(\ell)}{\sqrt{2}}\widehat{G}_q(\ell-p)  \widehat{F}_r(p)   dpd\ell dk \\
&= \frac{1}{2} \epsilon^2  \sum_{\substack{j,m,n,\\ q,r=1}}^2\iiint\displaylimits_{|k| < \delta}     \frac{i\omega_j(k)\omega_n(\ell)}{\phi^j_{mn}(k,\ell)} \frac{\widehat{\vartheta}(p)}{\widehat{\vartheta}(k)} \conj{\widehat{F}_j(k)} \widehat{G}_m^c(k-\ell) \widehat{G}_q(\ell-p)  \widehat{F}_r(p) dpd\ell dk
\end{aligned}\]

plus its complex conjugate. We can restrict to $\widehat{G}_q^c(\ell-p)$ since those terms with $\widehat{G}_q^s$ are $\mathcal{O}(\epsilon^3)$ with only one factor of $\widehat{\vartheta}^{-1}(k).$ Thus we have 
\begin{equation}
\frac{1}{2} \epsilon^2  \sum_{\substack{j,m,n,\\ q,r=1}}^2\iiint\displaylimits_{|k| < \delta}     \frac{i\omega_j(k)\omega_n(\ell)}{\phi^j_{mn}(k,\ell)} \frac{\widehat{\vartheta}(p)}{\widehat{\vartheta}(k)} \conj{\widehat{F}_j(k)} \widehat{G}_m^c(k-\ell) \widehat{G}^c_q(\ell-p)  \widehat{F}_r(p) dpd\ell dk . \label{split}
\end{equation}

Using the support of $\widehat{G}^c,$ we can split this integral up into two different regions depending on the value of $p$. Since $|k| < \delta,$ the factor of $\widehat{G}_m^c(k-\ell) $ above requires that $\ell \approx \pm k_0.$ Then similarly, the factor of $\widehat{G}^c_q(\ell-p)$ requires that $p \approx 0$ or $p \approx \pm 2k_0.$ More precisely, if we define 
\begin{equation}
\widehat{G}^c = \sum_{\{l=\pm\}}\widehat{G}^{c,l}
\label{plusminusg} \end{equation} 
with $\text{supp}( \widehat{G}^{c,l}(k)  )  \subset \{|k-lk_0| < \delta\}$, then one of those terms becomes 

\[\begin{aligned}
 &\epsilon^2   \iiint\displaylimits_{\substack{|k|<\delta\\ |p|<3 \delta}}    \frac{i\omega_j(k)\omega_n(\ell)}{\phi^j_{mn}(k,\ell)} \frac{\widehat{\vartheta}(p)}{\widehat{\vartheta}(k)} \conj{\widehat{F}_j(k)} \widehat{G}_m^c(k-\ell) \widehat{G}^c_q(\ell-p)  \widehat{F}_r(p)  \;dpd\ell dk  \\
& = \epsilon^2  \int \bigg(\int\displaylimits_{|k| < \delta}    \frac{i\omega_j(k)\omega_n(\ell)}{\phi^j_{mn}(k,\ell)\widehat{\vartheta}(k)}  \conj{\widehat{F}_j(k)} \widehat{G}_m^c(k-\ell) dk \bigg)\\
&\hspace{2 in} \times \bigg(\int\displaylimits_{|p| < 3\delta}  (\widehat{\vartheta}(p)-\widehat{\vartheta}(\ell - nk_0)) \widehat{G}^{c,n}_q(\ell-p)  \widehat{F}_r(p)  \;dp\bigg)  d\ell \\
& + \epsilon^2  \int\bigg(  \int\displaylimits_{|k| < \delta}   \frac{i\omega_j(k)\omega_n(\ell)}{\phi^j_{mn}(k,\ell)\widehat{\vartheta}(k)}  \left(\widehat{\vartheta}(\ell- nk_0)-\widehat{\vartheta}(k - (l+n)k_0)\right) \conj{\widehat{F}_j(k)} \widehat{G}^{c,l}_m(k-\ell) dk \bigg)\\
&\hspace{2 in} \times \bigg(\int\displaylimits_{|p| < 3\delta} \widehat{G}^{c,n}_q(\ell-p)  \widehat{F}_r(p)  \;dp\bigg)  d\ell  \\
& + \epsilon^2  \int\bigg(  \int\displaylimits_{|k| < \delta}   \frac{i\omega_j(k)\omega_n(\ell)}{\phi^j_{mn}(k,\ell)\widehat{\vartheta}(k)} \widehat{\vartheta}(k -(l+n)k_0) \conj{\widehat{F}_j(k)} \widehat{G}^{c,l}_m(k-\ell) dk \bigg)\\
&\hspace{2 in} \times \bigg(\int\displaylimits_{|p| < 3\delta}\widehat{G}^{c,n}_q(\ell-p)  \widehat{F}_r(p)  \;dp\bigg)  d\ell  \\
\end{aligned}\]

The first and second term can each be bounded using Proposition \ref{prop:lemma9} from the Appendix. If $l = -n$ the third term can be bounded directly since $\widehat{\vartheta}(k -(l+n)k_0) = \widehat{\vartheta}(k).$ If $l =n,$ we will need to integrate by parts with respect to $t$ as we did for $E_1.$ This calculation is motivated by the second normal form transformation of \cite{CEW}. We move to a rotating coordinate frame and we integrate with respect to $t$ giving us

\[
 \epsilon^2  \sum_{\substack{j,m,n,\\ q,r=1}}^2\int_0^t \hspace{-.1 in} \iiint\displaylimits_{\substack{|k|<\delta\\ |p \pm 2k_0|<3 \delta}}  \frac{i\omega_j(k)\omega_n(\ell)}{\phi^j_{mn}(k,\ell)} \frac{\widehat{\vartheta}(k \pm 2k_0)}{\widehat{\vartheta}(k)} e^{i\phi^j_{mqr}(k,\ell,p)s} \conj{\widehat{f}_j(k)} \widehat{g}^c_m(k-\ell) \widehat{g}_q(\ell-p)  \widehat{f}_r(p)  \;dpd\ell dk ds
\]
with 
\[ \phi^j_{mqr}(k,\ell,p) =   -\omega_j(k)+ \omega_m(k-\ell) + \omega_q(\ell-p) + \omega_r(p)\]

Since $\phi^j_{mqr}(k,\ell,p) \neq 0$ in the appropriate region, we can integrate by parts with respect to $s$ to get

\[\begin{aligned}
 &\epsilon^2  \sum_{\substack{j,m,n,\\ q,r=1}}^2\int_0^t  \iiint  \frac{i\omega_j(k)\omega_n(\ell)}{i\phi^j_{mn}(k,\ell)\phi^j_{mqr}(k,\ell,p)} \frac{\widehat{\vartheta}(k \pm 2k_0)}{\widehat{\vartheta}(k)} e^{i\phi^j_{mqr}(k,\ell,p)s}\\
&\hspace{2 in}\times\partial_s\left( \conj{\widehat{f}_j(k)} \widehat{g}^c_m(k-\ell) \widehat{g}_q(\ell-p)  \widehat{f}_r(p)  \right)\;dpd\ell dk ds
\end{aligned}\]
plus a boundary term. Both can then be bounded directly.


For $\mathcal{O}(\epsilon^3)$ we have
\[ \begin{aligned}
&2 \epsilon^3 \sum_{j=1}^2  \int\displaylimits_{|k|< \delta}\conj{B_j( \widehat{G}^c, \widehat{\vartheta}^{-1} N(\widehat{G}, \widehat{\vartheta} \widehat{F})   )}\cdot B_j(\widehat{G}^c, \widehat{F}) \\
&=2 \epsilon^3 \sum_{\substack{j,m,n,q,\\ r,a,b=1}}^2 \iiiint\displaylimits_{|k| < \delta} \conj{ \frac{\omega_j(k)/\sqrt{2}}{\phi^j_{mn}(k,\ell)} \frac{\widehat{\vartheta}(\ell)}{\widehat{\vartheta}(k)} \widehat{G}_m^c(k-\ell) \frac{\widehat{\vartheta}(p)}{\widehat{\vartheta}(\ell)}\frac{i\omega_n(\ell)}{\sqrt{2}} \widehat{G}_q(\ell-p)\widehat{F}_r(p)      }\\
&\hspace{1 in }\times \frac{\omega_j(k)/\sqrt{2}}{\phi^j_{ab}(k,q)} \frac{\widehat{\vartheta}(q)}{\widehat{\vartheta}(k)} \widehat{G}^c_a(k-q) \widehat{F}_b(q)\;dqd\ell dp dk
\end{aligned}\]
along with the complex conjugate. Again we restrict to  $\text{supp}(\widehat{G}^c)$ giving us 
\[\begin{aligned}
-\frac{\sqrt{2}}{2} \epsilon^3 \sum_{\substack{j,m,n,q,\\ r,a,b=1}}^2 & \iiiint\displaylimits_{|k| < \delta} \frac{i\omega_j(k)\omega_j(k)\omega_n(\ell)}{\phi^j_{mn}(k,\ell)\phi^j_{ab}(k,q)} \frac{\widehat{\vartheta}(p)\widehat{\vartheta}(q)}{\widehat{\vartheta}(k)\widehat{\vartheta}(k)}      \\
&\hspace{1 in }\times  \conj{\widehat{G}_m^c(k-\ell)  \widehat{G}^c_q(\ell-p)\widehat{F}_r(p) }\widehat{G}^c_a(k-q) \widehat{F}_b(q)\;dqd\ell dp dk
\end{aligned}\]
This term can be bounded exactly as (\ref{split}) using the different supports of $\ell$ and $p.$


\section{Error Estimates} 

In this section, we will put together the bounds found in each of the separate regions. Recall, that the goal is to show that the error is $\mathcal{O}(1)$ bounded in terms of epsilon for $t \in \left[0,\frac{T_0}{\epsilon^2}\right].$ This is done by showing
\[  \partial_t E \lesssim \epsilon^2(1+E) + \epsilon^3E^2  \]
which by using Gronwall's inequality will show
\[  \sup_{t \in [0,T_0/\epsilon^2]} E\left(t\right) \leq C \]
for a constant independent of $\epsilon.$
It was first shown that, with the addition of $E_1,$ all terms of $\partial_t E$ could be bounded by some power of $\epsilon$ times some power of $E.$ In other words, the quasilinearity no longer prevented the estimates from closing. To get the bound for the time interval necessary we then used the space-time resonance method. We integrated both sides with respect to $t$ to get 
\[  E(t)-E(0) =   \int_0^t \partial_s E(s) \; ds .\]
To bound the right-hand side we integrated by parts with respect to $s$ and bounded the remaining terms. This gave us 
\[  E(t)-E(0) \lesssim  \epsilon \left(E(t) - E(0)\right) +   \int_0^t\epsilon^2(1+E(s)) + \epsilon^3E(s)^2 \; ds ,\]
and so 
\[(1-C\epsilon)  E(t)  \lesssim (1-C\epsilon)E(0) +   \int_0^t\epsilon^2(1+E(s)) + \epsilon^3E(s)^2 \; ds.  \]

In essence the boundary terms have been added to the energy as $\mathcal{O}(\epsilon)$ correction terms. Since they are small and no longer integrated with respect to $t,$ these terms do not effect the fact that the left-hand side above is equivalent to $\|R\|_{H^2}.$ We then have

\[ 
E\left(t\right) \leq C_1 \left(  E(0) +   \int_0^t\epsilon^2(1+E(s)) + \epsilon^3E(s)^2 \; ds \right)    
\] 
and so for all $t$ such that $\epsilon E(t) \leq 1,$
\[ \begin{aligned}
E\left(t\right) &\leq C_1  E(0) +  C_1\epsilon^2  \int_0^t (1+2E(s)) \; ds \\   
&= C_1  (E(0)+\epsilon^2t) + 2C_1 \epsilon^2  \int_0^t E(s) \; ds.
\end{aligned}\] 
An application of Gronwall's inequality then gives us
\[ 
E(t) \leq C_1\left(E(0) + \epsilon^2t\right) e^{2C_1\epsilon^2t}.
\]
For $t=T_0/\epsilon^2$ we have
\[ 
E\left(T_0/\epsilon^2\right) \leq C_1\left(E(0) + T_0\right) e^{2C_1T_0}.
\]
Therefore if $\epsilon$ is small enough such that 
\[  C_1\left(E(0) + T_0\right) e^{2C_1T_0} \leq \frac{1}{\epsilon}, \]
then we have
\[  \sup_{t \in [0,T_0/\epsilon^2]} \|R(t)\|_{H^2} < C\]
independent of $\epsilon$ as desired.


\section{Appendix}

We include here the proof of some of the estimates used throughout this paper.

We first relate the effects of $\Omega$ to the more common operator $ \Lambda$, with symbol $\widehat{\Lambda^2}(k) = -|k|$,
and $\widehat{\Lambda} (k) = i\cdot \text{sgn}(k)\sqrt{|k|}$.

\begin{lemma}\label{lem:approx} For any $s > 0$, we have
\begin{eqnarray}
\| (\Omega^2 - \Lambda^2) u \|_{H^s} & \le &  C_s \| u \|_{L^2} \\ \nonumber
\| (\Omega - \Lambda ) u \|_{H^s} &\le & C_s  \| u \|_{L^2}
\end{eqnarray}
i.e. the difference between $(\Omega^2 - \Lambda^2)$ is infinitely smoothing, as is $(\Omega - \Lambda ) $.
\end{lemma}

\begin{proof}
We give the proof of the first of these inequalities.  The second is similar.
\begin{eqnarray}
\| (\Omega^2 - \Lambda^2) u \|_{H^s}^2 & = & \int(1+k^2)^s (-k \tanh(k) + |k|)^2 |\widehat{u}(k)|^2 dk \\ \nonumber
& \le & C_1 \int(1+k^2)^s e^{-2 |k|} |\widehat{u}(k)|^2 dk \le C_s \int |\widehat{u}(k)|^2 dk = C_s \| u \|_{L^2}^2\ ,
\end{eqnarray}
where the first inequality just used the fact that the hyperbolic tangent approaches its asymptotes exponentially fast. 
\qed\end{proof}

We now look at a few commutator arguments that will be useful when estimating $\partial_t E$ in $\mathcal{V}.$

\begin{proposition} \label{prop:comm1}  There exists a constant $C>0$ such that
\begin{equation}
\| [\Lambda, f]g \|_{L^2} \le C \| (1+ | \cdot |^{1/2}) \widehat{f}(\cdot) \|_{L^1} \| g  \|_{L^2}\ .
\end{equation}
\end{proposition}

\begin{proof}
We can write
\begin{eqnarray}
 && \widehat{\Lambda (fg)}(k)  - \widehat{f (\Lambda g)}(k) =  \int i (\text{sgn}(k)\sqrt{|k|} - \text{sgn}(\ell)\sqrt{|\ell||}) \widehat{f}(k-\ell) \widehat{g}(\ell) d\ell \\ \nonumber
 &&  = \int \left\{ \frac{(\text{sgn}(k)\sqrt{|k|} - \text{sgn}(\ell)\sqrt{|\ell||})}{1+|k-\ell |^{1/2}}  \right\} (1+ |k- \ell |^{1/2}) \widehat{f}(k-\ell)  \widehat{g}(\ell) d\ell \ .
\end{eqnarray}

We now use the following lemma.
\begin{lemma} \label{lemma:quotient}
\begin{equation}\label{eq:quotient}
\left| \frac{(\text{sgn}(k)\sqrt{|k|} - \text{sgn}(\ell)\sqrt{|\ell||})}{1+|k-\ell |^{1/2} }  \right| \le 3
\end{equation}
for all $\ell$ and $k$.
\end{lemma}

We prove the lemma below, but first note that from the lemma we have
\begin{equation}
 \widehat{\Lambda (fg)}(k)  - \widehat{f (\Lambda g)}(k) \le C \int |(1+ |k- \ell |^{1/2}) \widehat{f}(k-\ell)   \widehat{g}(\ell) | d\ell .
\end{equation}
Taking the $L^2$ norm of both sides and applying Young's inequality gives
\begin{equation}
\| \Lambda (fg) - f (\Lambda g) \|_{L^2} \le C  \| (1+ | \cdot |^{1/2}) \widehat{f}(\cdot) \|_{L^1} \|  \widehat{g}  \|_{L^2}\ .
\end{equation}

\qed\end{proof}

\begin{proof}[(Proof of Lemma \ref{lemma:quotient})]
We consider two cases:

Case 1: $k\ell>0.$ Since $\sqrt{|k|} = \sqrt{|k-\ell + \ell|} \leq \sqrt{|k-\ell|} + \sqrt{ |\ell|},$ we have \[ \left|  \frac{\sqrt{k}-\sqrt{\ell}}{1+ \sqrt{|k-\ell|}} \right| \leq 1\]

Case 2: $k\ell<0.$ Since $\text{sgn}(k) \neq \text{sgn}(\ell),$
\begin{equation}
\left| \frac{(\text{sgn}(k)\sqrt{|k|} - \text{sgn}(\ell)\sqrt{|\ell||})}{1+|k-\ell |^{1/2} }  \right| = \left| \frac{(\sqrt{|k|} + \sqrt{|\ell||})}{1+\sqrt{|k|+|\ell|} }  \right| \leq 3
\end{equation}

\qed\end{proof}

\begin{corollary} \label{prop:commcorr}
There exists a constant $C>0$ such that 
\begin{equation}
\| [\Omega,f] g \|_{L^2} \le C \| f \|_{H^{1+\delta}} \| g  \|_{L^2}\ 
\end{equation}
where $\delta > 0.$
\end{corollary}

\begin{proposition}  There exists a constant $C>0$ such that
\begin{equation}
\| \Omega^2 (fg) - f (\Omega^2 g) \|_{L^2} \le C \| (1+ | \cdot |) \widehat{f}(\cdot) \|_{L^1} \| g  \|_{L^2}\ .
\end{equation}
\end{proposition}

\begin{proposition} \label{prop:comm2} There exists a constant $C>0$ such that
\begin{equation}
\| \Lambda (\partial_x^2(gf) - g (\partial_x^2 f)) \|_{L^2} \le C \| (1+ | \cdot |^{5/2}) \widehat{g}(\cdot) \|_{L^1} \| f  \|_{H^{3/2}}\ .
\end{equation}
\end{proposition}

\begin{proof}
We can write
\begin{eqnarray}
 && \widehat{\Lambda \partial_x^2(gf))}(k)  -\widehat{\Lambda( g (\partial_x^2 f))}(k) =  \int i \text{sgn}(k)\sqrt{|k|}(k^2-\ell^2)  \widehat{g}(k-\ell) \widehat{f}(\ell) d\ell \\ 
 &&  \leq C \int  (\sqrt{|k-\ell|} + \sqrt{|\ell|})(k-\ell+2\ell)(k-\ell)  \widehat{g}(k-\ell) \widehat{f}(\ell) d\ell \\
&& \leq C \int (1+|k-\ell|^{5/2})\widehat{g}(k-\ell)(1+|\ell|^{3/2}) \widehat{f}(\ell) d\ell
\end{eqnarray}
Taking the $L^2$ norm of both sides and applying Young's inequality gives
\begin{equation}
\| \Lambda (\partial_x^2(gf) - g (\partial_x^2 f)) \|_{L^2} \le C \| (1+ | \cdot |^{5/2}) \widehat{g}(\cdot) \|_{L^1} \| f  \|_{H^{3/2}}\ .
\end{equation}
\qed\end{proof}

The following estimate is used to show $E(t)$ is equivalent to $\|R\|_{H^2}.$

\begin{proposition} \label{prop:modenergy} There exists a constant $C>0,$  $\gamma>0,$ and $r>2$ such that
\begin{align}
E_2 &= \sum_{j=1}^2 \int\displaylimits_{|k| < \delta}   \epsilon \conj{\widehat{F}_j} B_j(\widehat{G}^c, \widehat{F})+ \epsilon \widehat{F}_j\conj{ B_j(\widehat{G}^c, \widehat{F})} + 2\epsilon^2\conj{B_j(\widehat{G}^c, \widehat{F})}B_j(\widehat{G}^c, \widehat{F})\; dk \\
& \leq C \delta^{\gamma} \left(\|\widehat{A}\|_{L^r} + \|\widehat{A}\|_{L^r}^2 \right) \|F\|_{L^2}^2 \end{align}
\end{proposition}

\begin{proof}
Consider one term of the form
\[\begin{aligned}
&  \int\displaylimits_{|k| < \delta} \int_{\mathbb{R}} \conj{\widehat{F}_j(k)}\frac{\omega_j(k)}{\phi^{j}_{mn}(k,\ell)} \cdot \frac{\epsilon \delta}{\epsilon \delta + (1-\epsilon)|k|}    \widehat{G}_m^c(k-\ell)\widehat{F}_n(\ell) \;d\ell dk \\
&=   \iint_{\mathbb{R}^2} \chi_{[-\delta,\delta]}(k) \conj{\widehat{F}_j(k)}\frac{\omega_j(k)}{\phi^{j}_{mn}(k,\ell)} \cdot \frac{\epsilon \delta}{\epsilon \delta + (1-\epsilon)|k|}    \widehat{G}_m^c(k-\ell)\widehat{F}_n(\ell) \;d\ell dk \\
\end{aligned}\]

Apply Holder's inequality to bound this by 

\[    C \| |\widehat{F}_j| \cdot |h|\|_{L^p} \| \widehat{G}_m^c \ast \widehat{F}_n \|_{L^q }  \]

where $h(k) =  \chi_{[-\delta,\delta]}(k)    \frac{\epsilon \delta}{\epsilon \delta + (1-\epsilon)|k|}   $ and we have bounded  $ \left| \frac{\omega_j(k)}{\phi^{j}_{mn}(k,\ell)} \right| < C .$ Here $\frac{1}{p}+ \frac{1}{q} = 1$ and we chose $1 < p < 2$ (and hence $q>2$). Now apply Young's inequality to bound 

\begin{equation} \| \widehat{G}_m^c \ast \widehat{F}_n \|_{L^q }  \leq C  \| \widehat{G}_m^c \|_{L^r}     \| \widehat{F}_n \|_{L^2 }  \label{conv} \end{equation}
 
with $\frac{1}{r}= \frac{1}{q} + \frac{1}{2}.$ Now we need the following lemma.

\begin{lemma} \label{lemmaA}
 \begin{equation}\| \widehat{G}_m^c \|_{L^r} \leq C \epsilon^{\frac{1}{q} - \frac{1}{2}} \| \widehat{A}\|_{L^r} \label{gtoa}\end{equation}
\end{lemma}
(Here A is the profile of the approximate solution $\Psi.$ Note that since $q > 2,$ this grows as $\epsilon \rightarrow 0.$

We prove the lemma below, but now consider 
\[\begin{aligned}
\| |\widehat{F}_j| \cdot |h|\|_{L^p}^p &= \int  |\widehat{F}_j(k)|^p  |h(k)|^p dk\\
&\leq \left( \int  |\widehat{F}_j(k)|^{ps} dk \right)^{\frac{1}{s}} \left( \int  |h(k)|^{pt} dk\right)^{\frac{1}{t}}
\end{aligned}\]
with $\frac{1}{s} + \frac{1}{t} = 1,$ by Holder's inequality. Choose $ps = 2. $ Then 
\begin{equation}
\| |\widehat{F}_j| \cdot |h|\|_{L^p} \leq \| \widehat{F}_j(k)\|_{L^2} \left( \int  |h(k)|^{pt} dk\right)^{\frac{1}{pt}}. \label{cutoff}
\end{equation}
Now 
\[\begin{aligned}
 \int  |h(k)|^{pt} dk & = 2 \int_0^\delta \left(\frac{\epsilon \delta}{\epsilon \delta + (1-\epsilon) k}\right)^{pt} dk \\
& = 2 \int_0^\delta \left(\frac{1}{1 + (1-\epsilon) \left(\frac{k}{\epsilon \delta}\right)}\right)^{pt} dk \\
& = 2 \epsilon \delta \int_0^{1/\epsilon} \left( 1 + (1-\epsilon)x\right)^{-pt} dx.
\end{aligned}\]

Since $pt > 1,$ the integral over $x$ is convergent and can be bounded independent of $\epsilon$ if $0 < \epsilon < \frac{1}{2}.$ Thus 

\[   \left( \int  |h(k)|^{pt} dk\right)^{\frac{1}{pt}} \leq C (\epsilon \delta)^{\frac{1}{pt}} \leq C \delta^{\frac{1}{pt}} \epsilon^{\frac{1}{pt}}
  \]

Combining this with (\ref{conv})-(\ref{cutoff}) we have 

\[\begin{aligned}
\bigg| \int\displaylimits_{|k| < \delta} \int_{\mathbb{R}} &\conj{\widehat{F}_j(k)}\frac{\omega_j(k)}{\phi^{j}_{mn}(k,\ell)} \cdot \frac{\epsilon \delta}{\epsilon \delta + (1-\epsilon)|k|}    \widehat{G}_m^c(k-\ell)\widehat{F}_n(\ell) \;d\ell dk \bigg|\\
& \leq  C \delta^{\frac{1}{pt}} \epsilon^{\frac{1}{pt}+ \frac{1}{q} -\frac{1}{2}} \|\widehat{A}\|_{L^r} \| \widehat{F}\|^2_{L^2}
\end{aligned}\]

Now recall that $\frac{1}{t} = 1 - \frac{1}{s}$ so 
\[  \frac{1}{pt} = \frac{1}{p} - \frac{1}{ps} = \frac{1}{p} - \frac{1}{2} \]
since $ps = 2.$ Thus, 
\[ \frac{1}{pt}+ \frac{1}{q} -\frac{1}{2}= \left( \frac{1}{p} - \frac{1}{2} \right) + \frac{1}{q} -\frac{1}{2} = \left( \frac{1}{p} + \frac{1}{q}    \right) - \frac{1}{2} - \frac{1}{2} = 0. \]

Thus, we have
\[ \bigg| \int\displaylimits_{|k| < \delta} \int_{\mathbb{R}} \conj{\widehat{F}_j(k)}\frac{\omega_j(k)}{\phi^{j}_{mn}(k,\ell)} \cdot \frac{\epsilon \delta}{\epsilon \delta + (1-\epsilon)|k|}    \widehat{G}^c_m(k-\ell)\widehat{F}_n(\ell) \;d\ell dk \bigg| \leq C \delta^{\frac{1}{pt}} \|\widehat{A}\|_{L^r} \|F\|_{L^2}^2 \]

The second term in the energy can be bounded similarly. \qed

\end{proof}

\begin{proof}[(Proof of Lemma \ref{lemmaA})] Taking the Fourier transform of $G_m^c$ we find that,
\[ \widehat{G}_m^c(k) =  \frac{1}{\epsilon}\widehat{A}\left(\frac{k-k_0}{\epsilon}\right) e^{i\omega_0 t}e^{ic_g(k-k_0)t}  \]
plus a term localized around $-k_0.$ Therefore,
\[\begin{aligned} \| \widehat{G}_m^c \|_{L^r}^r & \leq C \int \frac{1}{\epsilon^r} \left|\widehat{A}\left(\frac{k-k_0}{\epsilon}\right)\right|^r dk\\
& = \epsilon^{1-r} \int \left|\widehat{A}\left(p\right)\right|^r dp	\\
& = \epsilon^{1-r} \| \widehat{A}\|_{L^r}.
\end{aligned}\]
The lemma then follows by taking the $r^{th}$ roots of both sides and using 
\[\frac{1}{r}-1 = \left(\frac{1}{q} + \frac{1}{2}\right) - 1 = \frac{1}{q} - \frac{1}{2}.\] \qed

\end{proof}
The next proposition and its proof are similar to Lemma 9 in \cite{CEW} but for clarity we state the lemma for the specific case needed in this paper.

\begin{proposition} \label{prop:lemma9}  Assume that $G^{c,n}$ is defined from the approximation (\ref{plusminusg}) and that $F \in L^2.$ Then there exists $C>0$ such that
\[  \left\| \int (\widehat{\vartheta}(\ell)-\vartheta(\cdot - nk_0)) \widehat{G}^{c,n}(\cdot-\ell)  \widehat{F}(\ell)  \;d\ell \right\|_{L^2} \leq C \epsilon \|F\|_{L^2}
\]
with $\widehat{\vartheta} (k)$ defined as in (\ref{vartheta}).
\end{proposition}

\begin{proof}

We calculate directly

\[ \begin{aligned}
& \left\| \int (\widehat{\vartheta}(\cdot - nk_0)-\widehat{\vartheta}(\ell)) \widehat{G}^{c,n}(\cdot-\ell)  \widehat{F}(\ell)  \;d\ell \right\|_{L^2}^2 \\
&= \int \left(\int (\widehat{\vartheta}(k - nk_0)-\widehat{\vartheta}(\ell)) \widehat{G}^{c,n}(k-\ell)  \widehat{F}(\ell)  \;d\ell \right)^2 dk \\
&= \int \left(\int (\widehat{\vartheta}(k - nk_0)-\widehat{\vartheta}(\ell)) \frac{1}{\epsilon}\widehat{A}\left(\frac{k-\ell-nk_0}{\epsilon}\right) e^{i\omega_0 t}e^{ic_g(k-\ell-nk_0)t} \widehat{F}(\ell)  \;d\ell \right)^2 dk \\
&\leq \int \left( C_{\vartheta} \int \left| \frac{(k - nk_0)-\ell }{\epsilon}\right| \left|\widehat{A}\left(\frac{k-\ell-nk_0}{\epsilon}\right)\right|\left| \widehat{F}(\ell)\right|  \;d\ell \right)^2 dk \\
&\leq C_{\vartheta}^2 \int \left(  \int \left|\frac{p }{\epsilon}\right| \left|\widehat{A}\left(\frac{p}{\epsilon}\right)\right| \;dp \right)^2 \|F\|_{L^2}^2  \\
&\leq C_{\vartheta}^2 \int \left(  \epsilon\int \left|m\right| \left|\widehat{A}\left(m\right)\right| \;dm \right)^2 \|F\|_{L^2}^2  \\
&\leq C \epsilon^2 \|F\|_{L^2}^2
\end{aligned}\]

We used the fact that $\widehat{\vartheta}$ is Lipschitz in the first inequality, Young's inequality in the next inequality,  a substitution in the next, and finally the fact that $\|A\|_{H^{6}} \leq C_1.$\qed

\end{proof}

\section{Acknowledgements} 

The authors' research was supported in part by the NSF through grant DMS-1311553.  CEW also thanks G. Schneider for many discussions about the derivation and justification of amplitude and modulation equations.

\bibliographystyle{abbrv}
\bibliography{MyBib}{}

\end{document}